\documentclass{amsart}

\usepackage{kotex}
\usepackage{amsmath}
\usepackage{amsfonts}
\usepackage{amscd}
\usepackage{stmaryrd}
\usepackage{slashed}

\usepackage{amssymb}
\usepackage{graphicx}
\usepackage{caption}
\usepackage{subcaption}
\usepackage{dsfont}
\usepackage{color}
\usepackage{hyperref}
\usepackage{bbm}
\usepackage{a4wide}
\usepackage{sseq}
\usepackage{tikz-cd}
\usepackage[abs]{overpic} 

\newtheorem{theorem}{Theorem}[section]
\newtheorem{lemma}[theorem]{Lemma}
\newtheorem{proposition}[theorem]{Proposition}
\newtheorem{corollary}[theorem]{Corollary}

\newtheorem{example}[theorem]{Example}

\theoremstyle{definition}
\newtheorem*{definition*}{Definition}
\newtheorem{definition}[theorem]{Definition}

\theoremstyle{remark}
\newtheorem*{remark*}{Remark}
\newtheorem{remark}[theorem]{Remark}

\theoremstyle{conjecture}
\newtheorem*{conjecture*}{Conjecture}

\numberwithin{equation}{section}

\newcommand{\myproof}[2]{Proof of {#1} {#2}}

\begin{document}

\title[Hidden $\operatorname{Sp}(1)$-symmetry and Brane Quantization on HyperK\"ahler Manifolds]{Hidden $\operatorname{Sp}(1)$-symmetry and Brane Quantization on HyperK\"ahler Manifolds}

\author{NaiChung Conan Leung, AND YuTung Yau}
\address{The Institute of Mathematical Sciences and Department of Mathematics, The Chinese University of Hong Kong, Shatin, Hong Kong}
\email{leung@math.cuhk.edu.hk}
\address{The Institute of Mathematical Sciences and Department of Mathematics, The Chinese University of Hong Kong, Shatin, Hong Kong}
\email{ytyau@math.cuhk.edu.hk}

\thanks{}

\maketitle

\begin{abstract}
	For a fixed prequantum line bundle $L$ over a hyperK\"ahler manifold $X$, we find a natural $\operatorname{Sp}(1)$-action on $\Omega^*(X, L)$ intertwining a twistor family of $\operatorname{Spin}^{\operatorname{c}}$-Dirac Laplacians on the spaces of $L$-valued $(0, *)$-forms on $X$, noting that $L$ is holomorphic for only one complex structure in the twistor family. This establishes a geometric quantization of $X$ via Gukov-Witten brane quantization and leads to a proposal of a mathematical definition of $\operatorname{Hom}(\overline{\mathcal{B}}_{\operatorname{cc}}, \mathcal{B}_{\operatorname{cc}})$ for the canonical coisotropic A-brane $\mathcal{B}_{\operatorname{cc}}$ on $X$ and its conjugate brane $\overline{\mathcal{B}}_{\operatorname{cc}}$.
\end{abstract}

\section{Introduction}
A hyperK\"ahler manifold $X$ admits a family $\{J_\zeta\}_{\zeta \in \mathbb{S}^2}$ of complex structures parametrized by the unit sphere $\mathbb{S}^2$ in the Lie algebra $\mathfrak{sp}(1)$, and $\operatorname{Sp}(1)$ acts on $\mathbb{S}^2$ by its adjoint action on $\mathfrak{sp}(1)$. There induces an $\operatorname{Sp}(1)$-action on $\Omega^*(X, \mathbb{C})$ intertwining the Dolbeault Laplacians $\Delta_{\overline{\partial}_{J_\zeta}}$ (which are indeed the same), but not $\overline{\partial}_{J_\zeta}$. Verbitsky generalizes this to the case of $\Omega^*(X, E)$ for any \emph{hyperholomorphic bundle} $E$ on $X$ \cite{Ver1996} (see Theorem \ref{Theorem 3.1}), which is a Hermitian vector bundle over $X$ with a unitary connection whose curvature $F_E$ is of type $(1, 1)$ with respect to any $J_\zeta$ for $\zeta \in \mathbb{S}^2$, or equivalently, $F_E$ is an $\operatorname{Sp}(1)$-invariant form in $\Omega^2(X, \operatorname{End}(E))$. Unlike the previous case, $\Delta_{\overline{\partial}_{J_\zeta}}$ are in general not the same for different $\zeta$'s in $\mathbb{S}^2$.\par
In the present paper, instead, we consider a Hermitian vector bundle $E$ over $X$ with a unitary connection of curvature $F_E$ proportional to the K\"ahler form $\omega_J$ of a complex structure $J$ in the twistor family: $\tfrac{\sqrt{-1}}{2\pi} F_E = \omega_J \operatorname{Id}_E$. Note that $E$ is only $\pm J$-holomorphic, but it still admits an $\mathbb{S}^2$-indexed family of $\operatorname{Spin}^{\operatorname{c}}$-Dirac operators $\slashed{D}_{J_\zeta}$ on $\Omega_{J_\zeta}^{0, *}(X, E)$ induced by $J_\zeta$. Surprisingly, even though $\omega_J$ is not $\operatorname{Sp}(1)$-invariant, we still find an $\operatorname{Sp}(1)$-symmetry intertwining the Dirac Laplacian operators $\slashed{D}_{J_\zeta}^2$ on $\Omega_{J_\zeta}^{0, *}(X, E)$. More precisely, we have proved the following main result of this paper.

\begin{theorem}
	\label{Theorem 1.1}
	Let $X$ be a hyperK\"ahler manifold and $E$ be a Hermitian vector bundle over $X$ with a unitary connection of curvature $F_E$ with $\tfrac{\sqrt{-1}}{2\pi}F_E = \omega_J \operatorname{Id}_E$. Then there is an $\operatorname{Sp}(1)$-action
	\begin{align*}
		\chi: \operatorname{Sp}(1) \times \Omega^*(X, E) \to \Omega^*(X, E), \quad (\eta, s) \mapsto \eta \cdot s,
	\end{align*}
	such that for all $\eta \in \operatorname{Sp}(1)$ and $\zeta \in \mathbb{S}^2$, the following diagram commutes:
	\begin{equation*}
	\begin{tikzcd}
	\Omega_{J_\zeta}^{0, *}(X, E) \ar[d, "\eta \cdot"'] \ar[r, "\slashed{D}_{J_\zeta}^2"] & \Omega_{J_\zeta}^{0, *}(X, E) \ar[d, "\eta \cdot"]\\
	\Omega_{J_{\eta \cdot \zeta}}^{0, *}(X, E) \ar[r, "\slashed{D}_{J_{\eta \cdot \zeta}}^2"'] & \Omega_{J_{\eta \cdot \zeta}}^{0, *}(X, E)
	\end{tikzcd}
	\end{equation*}
	where $\slashed{D}_{J_\zeta}$ is the $\operatorname{Spin}^{\operatorname{c}}$-Dirac operator on $\Omega_{J_\zeta}^{0, *}(X, E)$ defined as in Definition \ref{Definition 3.2}.
\end{theorem}

Note that the above theorem does not require the assumption that $X$ is compact. Also, the $\operatorname{Sp}(1)$-symmetry $\chi$ in this theorem is different from the one for hyperholomorphic bundles where $\mathbb{Z}$-gradings are preserved - only $\mathbb{Z}_2$-gradings are preserved in Theorem \ref{Theorem 1.1}. The proof of Theorem \ref{Theorem 1.1} relies on Verbitsky's results \cite{Ver1990, Ver1996, Ver2002, Ver2007} in hyperK\"ahler geometry, especially that there is an $\operatorname{Sp}(4, \mathbb{C})$-action on the space $\Omega^*(X, E)$ of $E$-valued forms on $X$. We need to find out an appropriate subgroup action of $\operatorname{Sp}(1)$ and apply the Lichnerowicz formula to prove Theorem \ref{Theorem 1.1}. Our idea of proof can also be used to reprove Theorem \ref{Theorem 3.1}.\par
A direct consequence of Theorem \ref{Theorem 1.1} is a vanishing theorem on an $\mathbb{S}^2$-indexed family of $\operatorname{Spin}^{\operatorname{c}}$-Dirac operators, and its proof is due to Kodaira Vanishing Theorem.

\begin{corollary}
	\label{Corollary 1.2}
	With the same assumption as in Theorem \ref{Theorem 1.1}, for all $\zeta \in \mathbb{S}^2$, we have
	\begin{align*}
	\ker \slashed{D}_{J_\zeta}^- = 0,
	\end{align*}
	where $\slashed{D}_{J_\zeta}^-$ is the restriction of $\slashed{D}_{J_\zeta}$ on the odd degree component $\Omega_{J_\zeta}^{0, 2*+1}(X, E)$.
\end{corollary}

The index of $\slashed{D}_{J_\zeta}$ is defined as the formal difference $\operatorname{ind} \slashed{D}_{J_\zeta} := \ker \slashed{D}_{J_\zeta}^+ - \ker \slashed{D}_{J_\zeta}^-$, where $\slashed{D}_{J_\zeta}^+$ is the restriction of $\slashed{D}_{J_\zeta}$ on the even degree component $\Omega_{J_\zeta}^{0, 2*}(X, E)$. Hence, this corollary implies that $\operatorname{ind} \slashed{D}_{J_\zeta}$ is an honest vector space and the family $\{ \operatorname{ind} \slashed{D}_{J_\zeta} \}_{\zeta \in \mathbb{S}^2}$ is $\operatorname{Sp}(1)$-equivariant.\par
While there are similarities between Theorem \ref{Theorem 3.1} proved by Verbitsky and Theorem \ref{Theorem 1.1}, we find a unified physical interpretation of them. In particular, Theorem \ref{Theorem 1.1} leads to a proposed definition of the morphism spaces for certain pairs of coisotropic A-branes.

\subsection{Studying morphism spaces between coisotropic A-branes via $\operatorname{Sp}(1)$-symmetry}
\label{Subsection 1.1}
\quad\par
\emph{Coisotropic A-branes} are discovered by Kapustin-Orlov \cite{KapOrl2003} as natural boundary conditions in A-model and are studied so as to understand mirror symmetry. Since then, a number of attempts to give morphism spaces between coisotropic A-branes a rigorous mathematical definition have been emerging, e.g. \cite{AldZas2005, BisGua2022, KapLi2005, KapOrl2004, Qin2020, Qui2012}. While this problem remains challenging, these morphism spaces are found to be important for other areas as well. Considering A-model for $(X, \omega_K)$, Gukov-Witten \cite{GukWit2009} realized that $\operatorname{Hom}_{\omega_K}(\mathcal{B}_{\operatorname{cc}}, \mathcal{B}_{\operatorname{cc}})$ and $\operatorname{Hom}_{\omega_K}(\mathcal{B}, \mathcal{B}_{\operatorname{cc}})$ are related to deformation quantization and geometric quantization respectively, where $\mathcal{B}_{\operatorname{cc}}$ is the space-filling brane on $X$ with Chan-Paton bundle $L$ such that $c_1(L) = [\omega_J]$ and $\mathcal{B}$ is a brane supported on a $J$-holomorphic Lagrangian submanifold $M$ of $X$ with trivial Chan-Paton bundle, leading to the work of Bischoff-Gualtieri \cite{BisGua2022} on defining $\operatorname{Hom}_{\omega_K}(\mathcal{B}, \mathcal{B}_{\operatorname{cc}})$. The works of Gaiotto-Witten \cite{GaiWit2022} and Etingof-Frenkel-Kazhdan \cite{EtiFreKaz2021} reveal that $\operatorname{Hom}_{\omega_K}(\overline{\mathcal{B}}_{\operatorname{cc}}, \mathcal{B}_{\operatorname{cc}})$ is related to the analytic geometric Langlands program, where $\overline{\mathcal{B}}_{\operatorname{cc}}$ is the brane with Chan-Paton bundle $L^\vee$. The $(\overline{\mathcal{B}}_{\operatorname{cc}}, \mathcal{B}_{\operatorname{cc}})$-system was studied by Aldi-Zaslow \cite{AldZas2005} in the special case when $X$ is a $4$-dimensional affine torus.\par
While hyperholomorphic bundles over hyperK\"ahler submanifolds of $X$ are regarded as \emph{(B, B, B)-branes} in physics, the aforementioned branes $\mathcal{B}_{\operatorname{cc}}, \overline{\mathcal{B}}_{\operatorname{cc}}, \mathcal{B}$ are all \emph{(A, B, A)-branes} \cite{GukWit2009}, i.e. objects which are simultaneously branes with respect to $\mathcal{J}_\zeta$ for all $\zeta \in \mathbb{S}^2$, where $\{\mathcal{J}_\zeta\}_{\zeta \in \mathbb{S}^2}$ is the $\mathbb{S}^2$-indexed family of generalized complex structures determined by the triple $(\omega_I, J, \omega_K)$ \cite{Gua2011, HonSti2015} (see Example \ref{Example A.8}). It is explained physically in Section 2.3 in \cite{GukWit2009} (and also Section 3.9 in \cite{GaiWit2022}) that for $X$ being hyperK\"ahler and $M$ being compact, the morphism spaces for the $(\mathcal{B}, \mathcal{B}_{\operatorname{cc}})$-system with respect to $J$ and $\omega_K$ coincide. Indeed, (B, B, B)-branes and (A, B, A)-branes are two basic examples of \emph{hyperbranes} on $X$ in the sense of Definition \ref{Definition 4.1}, where $X$ is equipped with different generalized hyperK\"ahler structures arising from the single hyperK\"ahler structure. We expect that
\begin{itemize}
	\item \emph{for a pair of hyperbranes on a generalized hyperK\"ahler manifold such that the intersection of their supports is compact, there is an $\operatorname{Sp}(1)$-symmetry intertwining morphism spaces among them.}
\end{itemize}\par
Theorem \ref{Theorem 3.1} proved by Verbitsky verifies our expectation on space-filling (B, B, B)-branes, where their B-model morphism spaces are mathematically defined in terms of derived categories of coherent sheaves \cite{CalKatSha2003, KatSha2002} and computed by sheaf cohomologies. Morphism spaces for (A, B, A)-branes are more subtle. By Proposition 3.27 in \cite{Gua2011} (see also \cite{HonSti2015}), $\mathcal{J}_{\pm \mathbf{j}}$ corresponds to a B-model and apart from $\zeta = \pm \mathbf{j}$, $\mathcal{J}_\zeta$ is the B-field transform of a generalized complex structure induced by a symplectic form on $X$, corresponding to an A-model. In general, we have no a priori mathematical definitions of the A-model morphism spaces for a pair of (A, B, A)-branes. For the $(\overline{\mathcal{B}}_{\operatorname{cc}}, \mathcal{B}_{\operatorname{cc}})$-system, we propose a definition (Definition \ref{Definition 4.2}) of its A-model morphism spaces. Corollary \ref{Corollary 1.2} can be interpreted as establishing an $\operatorname{Sp}(1)$-symmetry intertwining the $\mathbb{S}^2$-indexed family of morphism spaces for the $(\overline{\mathcal{B}}_{\operatorname{cc}}, \mathcal{B}_{\operatorname{cc}})$-system. Our expectation on certain other pairs of (A, B, A)-branes, including the $(\mathcal{B}_{\operatorname{cc}}, \mathcal{B}_{\operatorname{cc}})$-system, the $(\mathcal{B}, \mathcal{B}_{\operatorname{cc}})$-system and $(\mathcal{B}_0, \mathcal{B}_1)$-systems with supports of $\mathcal{B}_0, \mathcal{B}_1$ being $J$-holomorphic Lagrangian, will be discussed in Section \ref{Section 4}. Overall, we see that this $\operatorname{Sp}(1)$-symmetry is crucial as it helps us to understand morphism spaces between coisotropic A-branes via mathematically defined morphism spaces between B-branes.

\subsection{Brane quantization on hyperK\"ahler manifolds}
\quad\par
A further physical implication of our main result is relevant to Gukov-Witten brane quantization \cite{GukWit2009}, a new recipe to obtain quantum Hilbert spaces. Geometric quantization on a symplectic manifold $(M, \omega_M)$ requires a polarization to obtain a quantum Hilbert space $\mathcal{H}_{\operatorname{GQ}}$. Instead, brane quantization requires a \emph{complexification}, which is a set of data
\begin{equation*}
	(Y, \Omega, L, \tau),
\end{equation*}
where $(Y, \Omega)$ is a holomorphic symplectic manifold, $L$ is a Hermitian line bundle over $Y$ with a unitary connection of curvature $F_L$ with $\tfrac{\sqrt{-1}}{2\pi}F_L = \operatorname{Re} \Omega$ and $\tau: Y \to Y$ is an antiholomorphic involution with $\tau^*\Omega = \overline{\Omega}$ such that $M$ is embedded in $Y$ as a component of the fixed point set of $\tau$ with $\Omega \vert_M = \omega_M$ and the action of $\tau$ on $Y$ lifts to an action on $L$, restricting to the identity on $M$. In this recipe, a quantum Hilbert space is obtained by the A-model morphism space $\operatorname{Hom}_{\operatorname{Im}\Omega}( \mathcal{B}, \mathcal{B}_{\operatorname{cc}} )$, to be mathematically defined, where $\mathcal{B}, \mathcal{B}_{\operatorname{cc}}$ are A-branes on $(Y, \operatorname{Im}\Omega)$ defined similarly to those on the hyperK\"ahler manifold $X$. If the holomorphic symplectic structure on $Y$ extends to a hyperK\"ahler structure and $M$ is compact, then $\operatorname{Hom}_{\operatorname{Im}\Omega}( \mathcal{B}, \mathcal{B}_{\operatorname{cc}} )$ should be isomorphic to $\mathcal{H}_{\operatorname{GQ}}$ as vector spaces.\par
Our main result is related to geometric quantization on the hyperK\"ahler manifold $X$. $\operatorname{Sp}(1)$-symmetry for geometric quantization on $X$ was studied by Andersen-Malus\`a-Rembado in \cite{AndMalRem2022} under the assumption that there is an isometric $\operatorname{Sp}(1)$-action on $X$ which transitively permutes the K\"ahler forms. We, instead, assume $X$ is compact and $L$ is a Hermitian line bundle with a unitary connection of curvature $F_L$ with $\tfrac{\sqrt{-1}}{2\pi}F_L = \omega_J$. Following Kostant’s cohomological approach \cite{Kos1975}, the quantum Hilbert space of $(X, 2\omega_J)$ in the (pseudo)K\"ahler polarization $T_{\pm J}^{0, 1}X$ is given by
\begin{equation*}
\mathcal{H}_{\pm J} := H_{\overline{\partial}_{\pm J}}^{0, *}(X, L^{\otimes 2}).
\end{equation*}\par
The symplectic manifold $(X, 2\omega_J)$ admits a complexification as follows so that we can perform brane quantization. We take the product complex manifold $(\widehat{X}, \widehat{I}) = (X, -I) \times (X, I)$ equipped with the holomorphic symplectic form $\omega_{\widehat{J}} + \sqrt{-1} \omega_{\widehat{K}}$, where $\omega_{\widehat{J}} = \pi_1^*\omega_J + \pi_2^*\omega_J$, $\omega_{\widehat{K}} = -\pi_1^*\omega_K + \pi_2^*\omega_K$ and $\pi_1, \pi_2: \widehat{X} \to X$ are the projections on the first and the second factor respectively. Indeed, $\widehat{X}$ is hyperK\"ahler with the metric $\widehat{g} = \pi_1^*g + \pi_2^*g$. The required Hermitian line bundle with unitary connection is $\widehat{L} := L \boxtimes L$, and $X$ is embedded into $\widehat{X}$ as the diagonal. An involution on $\widehat{X}$ and its lifted action on $\widehat{L}$ are defined by swapping variables in the two factors.\par
This complexification $\widehat{X}$ admits two (A, B, A)-branes: the brane $\widehat{\mathcal{B}}$ supported on $X$ with trivial Chan-Paton bundle and the space-filling brane $\widehat{\mathcal{B}}_{\operatorname{cc}}$ with Chan-Paton bundle $\widehat{L}$. We can compare the $(\widehat{\mathcal{B}}, \widehat{\mathcal{B}}_{\operatorname{cc}})$-system with the $(\overline{\mathcal{B}}_{\operatorname{cc}}, \mathcal{B}_{\operatorname{cc}})$-system via Gaiotto-Witten \emph{folding trick} \cite{GaiWit2022}: this trick gives a physical expectation that
\begin{equation}
\label{Equation 1.1}
\operatorname{Hom}_{\omega_K}(\overline{\mathcal{B}}_{\operatorname{cc}}, \mathcal{B}_{\operatorname{cc}}) \cong \operatorname{Hom}_{\omega_{\widehat{K}}}(\widehat{\mathcal{B}}, \widehat{\mathcal{B}}_{\operatorname{cc}}).
\end{equation}
Indeed, $(\widehat{X}, \omega_{\widehat{K}})$ is the symplectic groupoid of $(X, \omega_K)$, whence (\ref{Equation 1.1}) also matches the proposal of defining morphisms between generalized complex branes by Bischoff-Gualtieri \cite{BisGua2022}. Eventually, adopting Definition \ref{Definition 4.2}, we have the following isomorphism by Hodge theory and Corollary \ref{Corollary 1.2},
\begin{equation*}
\operatorname{Hom}_{\omega_K}(\overline{\mathcal{B}}_{\operatorname{cc}}, \mathcal{B}_{\operatorname{cc}}) \cong \mathcal{H}_{\pm J}.
\end{equation*}
Thus, if the physical folding trick makes mathematical sense, then our main result verifies the relationship between geometric quantization and brane quantization on a hyperK\"ahler manifold.\par
\quad\par
The paper is organized as follows. In Section \ref{Section 2}, we shall review an $\operatorname{Sp}(4, \mathbb{C})$-symmetry on an arbitrary Hermitian vector bundle over $X$. In Section \ref{Section 3}, we shall be devoted to proofs of the theorems. In Section \ref{Section 4}, we shall introduce the notion of hyperbranes and discuss $\operatorname{Sp}(1)$-symmetries on morphism spaces between hyperbranes on a generalized hyperK\"ahler manifold in more details.

\subsection{Acknowledgement}
\quad\par
We thank Marco Gualtieri, Eric Sharpe, Misha Verbitsky and Eric Zaslow for useful comments and suggestions. This research was substantially supported by grants from the Research Grants Council of the Hong Kong Special Administrative Region, China (Project No. CUHK14306720 and CUHK14301721) and direct grants from the Chinese University of Hong Kong.

\section{$\operatorname{Sp}(4, \mathbb{C})$-Symmetry for a HyperK\"ahler Manifold}
\label{Section 2}
The proof of Theorem \ref{Theorem 1.1} requires our understanding on an $\operatorname{Sp}(4, \mathbb{C})$-symmetry arising from the hyperK\"ahler structure on $X$. In this section, we shall give a quick review on this symmetry and set up notations for proving our main theorem.\par
It was first discovered by Verbitsky \cite{Ver1990} that there is a complex representation of $\mathfrak{so}(5, \mathbb{C}) \cong \mathfrak{sp}(4, \mathbb{C})$ on $\Omega^*(X, \mathbb{C})$, noting that $\mathfrak{sp}(4, \mathbb{C})$ is the complexification of $\mathfrak{so}(4, 1) \cong \mathfrak{sp}(1, 1)$ (see also \cite{CaoZho2005, LeuLi2008}). Indeed, this representation comes from a fibrewise Lie algebra action and this can be generalized to the case of $\Omega^*(X, E)$ for a Hermitian vector bundle $E$: we have a fibrewise Lie algebra action $X \times \mathfrak{sp}(4, \mathbb{C}) \to \operatorname{End}(\textstyle\bigwedge T^*X_\mathbb{C} \otimes E)$, which integrates to a fibrewise Lie group action $X \times \operatorname{Sp}(4, \mathbb{C}) \to \operatorname{End}(\textstyle\bigwedge T^*X_\mathbb{C} \otimes E)$ as $\operatorname{Sp}(4, \mathbb{C})$ is simply connected.\par
To describe the above actions explicitly, we introduce the following operators on $\Omega^*(X, E)$. Recall that $\{J_\zeta\}_{\zeta \in \mathbb{S}^2}$ denotes the twistor family of complex structures on $X$. For each $\zeta \in \mathbb{S}^2$, define the operator $\operatorname{ad}(J_\zeta)$ on $\Omega^*(X, E)$ as $\operatorname{ad}(J_\zeta) s = (p-q)\sqrt{-1}s$ for all $p, q \in \mathbb{N}$ and sections $s \in \Omega_{J_\zeta}^{p, q}(X, E)$. Denote by $\omega_\zeta$ the K\"ahler form of $J_\zeta$. Let $L_{\omega_\zeta}$ be the Lefschetz operator associated to $\omega_\zeta$, acting on $\Omega^*(X, E)$, and $\Lambda_{\omega_\zeta}$ be the Hermitian adjoint of $L_{\omega_\zeta}$. Finally, define a linear operator $H$ on $\Omega^*(X, E)$ as $Hs = (k-2n)s$ for $k \in \mathbb{N}$ and $s \in \Omega^k(X, E)$, where $4n = \dim_\mathbb{R} X$. 

\begin{proposition}
	\label{Proposition 2.1}
	There is a complex representation $\rho^{\mathfrak{sp}(4, \mathbb{C})}$ of the complex Lie algebra $\mathfrak{sp}(4, \mathbb{C})$ on $\Omega^*(X, E)$, where $\mathfrak{sp}(4, \mathbb{C})$ is embedded into the $\mathbb{C}$-vector space of zeroth order differential operators on $\Omega^*(X, E)$ spanned by the following $10$ operators
	\begin{equation}
	\label{Equation 2.1}
	L_{\omega_I}, L_{\omega_J}, L_{\omega_K}, \Lambda_{\omega_I}, \Lambda_{\omega_J}, \Lambda_{\omega_K}, \operatorname{ad}(I), \operatorname{ad}(J), \operatorname{ad}(K), H.
	\end{equation}
	It integrates to a complex representation $\rho^{\operatorname{Sp}(4, \mathbb{C})}$ of the Lie group $\operatorname{Sp}(4, \mathbb{C})$ on $\Omega^*(X, E)$.
\end{proposition}

The induced actions of certain subgroups of $\operatorname{Sp}(4, \mathbb{C})$ are closely related to the main theorem. In Subsection \ref{Subsection 2.1}, we shall recall the $\operatorname{Sp}(1)$-representation on $\Omega^*(X, E)$ arising from the hypercomplex structure on $X$. In Subsection \ref{Subsection 2.2}, we shall describe the $\operatorname{SL}(2, \mathbb{C})$-representations on $\Omega^*(X, E)$ induced by anti-holomorphic symplectic forms on $X$. In Subsection \ref{Subsection 2.3}, we shall discuss the $\operatorname{Sp}(1)$-actions given by Clifford actions on $\Omega^*(X, E)$.

\subsection{The $\operatorname{Sp}(1)$-action induced by the hypercomplex structure}
\quad\par
\label{Subsection 2.1}
Of the Lie algebra $\mathfrak{sp}(1, 1)$, there is a Lie subalgebra $\mathfrak{sp}(1)$ generated by $\operatorname{ad}(I), \operatorname{ad}(J), \operatorname{ad}(K)$. The induced action of this copy of $\mathfrak{sp}(1)$ integrates to an $\operatorname{Sp}(1)$-action
\begin{equation}
	\label{Equation 2.2}
	\rho^{\operatorname{Sp}(1)}: \operatorname{Sp}(1) \to \operatorname{End}(\Omega^*(X, E))
\end{equation}
determined by the condition that for any $\zeta \in \mathbb{S}^2 \subset \mathfrak{sp}(1)$, $p, q \in \mathbb{N}$ and $s \in \Omega_{J_\zeta}^{p, q}(X, E)$,
\begin{align*}
\rho^{\operatorname{Sp}(1)}(\zeta)(s) = \sqrt{-1}^{p-q}s.
\end{align*}

\subsection{The $\operatorname{SL}(2, \mathbb{C})$-actions generated by anti-holomorphic symplectic forms}
\quad\par
\label{Subsection 2.2}
In this subsection, we discuss the $\operatorname{SL}(2, \mathbb{C})$-action on $\Omega^*(X, E)$ generated by an anti-holomorphic symplectic form with respect to a complex structure $J_\zeta$ for $\zeta \in \mathbb{S}^2$, say $J_\zeta = J$, and determine the decomposition of $\Omega^*(X, E)$ into irreducible $\operatorname{SL}(2, \mathbb{C})$-representations.\par
Define the $J$-holomorphic symplectic form $\Omega = \tfrac{1}{2}(\omega_K + \sqrt{-1}\omega_I)$. Let $L_{\overline{\Omega}}$ be the Lefschetz operator of $\overline{\Omega}$, $\Lambda_{\overline{\Omega}}$ be the Hermitian adjoint of $L_{\overline{\Omega}}$, and $H_{\overline{\Omega}} = [L_{\overline{\Omega}}, \Lambda_{\overline{\Omega}}]$. By Proposition 3.1 in \cite{Ver2007}, \footnote{In \cite{Ver2007}, it is assumed that $E$ is $J$-holomorphic, but this assumption is not necessary for proving that $(L_{\overline{\Omega}}, \Lambda_{\overline{\Omega}}, H_{\overline{\Omega}})$ forms a Lefschetz triple. Indeed, the proof is due to \cite{Fuj1987} and is analogous to the proof of the usual Lefschetz Theorem about the $\mathfrak{sl}(2, \mathbb{R})$-action.} $(L_{\overline{\Omega}}, \Lambda_{\overline{\Omega}}, H_{\overline{\Omega}})$ forms a Lefschetz triple acting on $\Omega_J^{*, *}(X, E)$ and for all $k \in \mathbb{N}$ and $s \in \Omega_J^{*, k}(X, E)$, $H_{\overline{\Omega}}s = (k-n)s$. The Lie algebra $\mathfrak{sl}(2, \mathbb{C})$ generated by $L_{\overline{\Omega}}, \Lambda_{\overline{\Omega}}, H_{\overline{\Omega}}$ over $\mathbb{C}$ lies in $\mathfrak{sp}(4, \mathbb{C})$. To determine the decomposition of $\Omega^*(X, E)$ into irreducibles, we give a definition of $\overline{\Omega}$-primitivity.

\begin{definition}
	A section $s \in \Omega^*(X, E)$ is said to be $\overline{\Omega}$-\emph{primitive} if $\Lambda_{\overline{\Omega}} s = 0$.
\end{definition}

\begin{remark}
	There are choices of $J$-holomorphic symplectic forms $\Omega^{(\mu)} := \tfrac{1}{2}(\omega_\mu + \sqrt{-1}\omega_{\mathbf{j}\mu})$ parametrized by unit quaternions $\mu \in \mathfrak{sp}(1)$ orthogonal to $\mathbf{j}$, and $\Omega^{(\mu)} = e^{\sqrt{-1}\theta} \Omega$ for some $\theta \in \mathbb{R}$. In particular, $\Omega^{(\mathbf{k})} = \Omega$. These give different choices of Lefschetz triples
	\begin{align*}
		(L_{\overline{\Omega}^{(\mu)}}, \Lambda_{\overline{\Omega}^{(\mu)}}, H_{\overline{\Omega}^{(\mu)}}) = (e^{\sqrt{-1}\theta} L_{\overline{\Omega}}, e^{-\sqrt{-1}\theta} \Lambda_{\overline{\Omega}}, H_{\overline{\Omega}}),
	\end{align*}
	but the resulting $\mathfrak{sl}(2, \mathbb{C})$-representation and $\overline{\Omega}^{(\mu)}$-primitivity are independent of the choice of $\mu$.
\end{remark}

For $p, q \in \mathbb{N}$, we denote by $\Omega_{J, \operatorname{prim}}^{p, q}(X, E)$ the vector subspace of $\overline{\Omega}$-primitive sections in $\Omega_J^{p, q}(X, E)$. For $m \in \mathbb{N}$, let $U_m$ be the $(m+1)$-dimensional irreducible complex representation of $\operatorname{SL}(2, \mathbb{C})$. Now we have the main proposition in this subsection.

\begin{proposition}
	\label{Proposition 2.4}
	There is an isomorphism of complex representations of $\operatorname{SL}(2, \mathbb{C})$,
	\begin{equation}
		\label{Equation 2.3}
		\Omega^*(X, E) \cong \bigoplus_{q=0}^n U_{n-q} \otimes \Omega_{J, \operatorname{prim}}^{*, q}(X, E),
	\end{equation}
	on the right hand side of which $\operatorname{SL}(2, \mathbb{C})$ acts on $\Omega_{J, \operatorname{prim}}^{*, q}(X, E)$ trivially.
\end{proposition}
\begin{proof}
	By Proposition 2.6 in \cite{Fuj1987}, the $\overline{\Omega}$-primitive decomposition on $\Omega^*(X, E)$ gives
	\begin{align*}
	\Omega^*(X, E) = \bigoplus_{q=0}^n \bigoplus_{i=0}^{n-q} L_{\overline{\Omega}}^i \Omega_{J, \operatorname{prim}}^{*, q}(X, E).
	\end{align*}
	Consider each $q \in \{0, ..., n\}$ and define $m = n - q$. The observation that for $s \in \Omega_{J, \operatorname{prim}}^{*, q}(X, E)$,
	\begin{equation}
	\label{Equation 2.4}
	\Lambda_{\overline{\Omega}} L_{\overline{\Omega}}^{i+1} s = (i+1)(m-i) L_{\overline{\Omega}}^i s \quad \text{and} \quad H_{\overline{\Omega}} L_{\overline{\Omega}}^i s = ( 2i - m ) L_{\overline{\Omega}}^i s,
	\end{equation}
	implies that $\bigoplus_{i=0}^m L_{\overline{\Omega}}^m \Omega_{J, \operatorname{prim}}^{*, q}(X, E)$ is an $\operatorname{SL}(2, \mathbb{C})$-subrepresentation of $\Omega^*(X, E)$. Define the $\mathbb{C}$-linear isomorphism
	\begin{equation*}
	\Phi: U_m \otimes \Omega_{J, \operatorname{prim}}^{*, q}(X, E) \to \bigoplus_{i=0}^{m} L_{\overline{\Omega}}^i \Omega_{J, \operatorname{prim}}^{*, q}(X, E)
	\end{equation*}
	as follows. We take the explicit model of $U_m$ that it is the vector space of homogeneous polynomials of degree $m$ in variables $x, y$ over $\mathbb{C}$. For each section $s \in \Omega_{J, \operatorname{prim}}^{*, q}(X, E)$ and $i \in \{0, ..., m\}$, define
	\begin{align*}
		\Phi( x^iy^{m-i} \otimes s ) = \tfrac{(m-i)!}{m!} L_{\overline{\Omega}}^i s.
	\end{align*}
	Then by (\ref{Equation 2.4}), we can see that for $s \in \Omega_{J, \operatorname{prim}}^{*, q}(X, E)$ and $f \in U_m$ with $fs := f \otimes s$,
	\begin{align*}
	L_{\overline{\Omega}} \Phi (fs) = \Phi (L_m f \otimes s), \quad \Lambda_{\overline{\Omega}} \Phi (fs) = \Phi (\Lambda_m f \otimes s), \quad \text{and} \quad H_{\overline{\Omega}} \Phi (fs) = \Phi (H_m f \otimes s).
	\end{align*}
	where $(L_m, \Lambda_m, H_m)$ is the Lefschetz triple given as in (\ref{Equation B.1}). We are done.
\end{proof}

\subsection{The $\operatorname{Sp}(1)$-actions induced by Clifford actions}
\quad\par
\label{Subsection 2.3}
The hypercomplex structure on $X$ gives rise to an $\mathbb{S}^2$-indexed family $\{\bigwedge T_{J_\zeta}^{*(0, 1)}X\}_{\zeta \in \mathbb{S}^2}$ of $\operatorname{Spin}^{\operatorname{c}}$-spinor bundles. Together with the hyperK\"ahler metric $g$ on $X$, there induces an $\mathbb{S}^2$-indexed family of fibrewise Clifford actions (parametrized by $\zeta \in \mathbb{S}^2$)
\begin{equation}
\label{Equation 2.5}
c_\zeta: \textstyle T^*X_\mathbb{C} \to \operatorname{End}(\textstyle\bigwedge T^* X_\mathbb{C} \otimes E)
\end{equation}
given as follows. For $\alpha \in T^*X_\mathbb{C}$, write $\alpha = \alpha^{1, 0} + \alpha^{0, 1}$ with $\alpha^{1, 0} \in T_{J_\zeta}^{1, 0}X$ and $\alpha^{0, 1} \in T_{J_\zeta}^{0, 1}X$. Then
\begin{equation*}
c_\zeta(\alpha) = \sqrt{2} (\alpha^{0, 1} \wedge - (g^{-1})^\sharp(\alpha^{1, 0}) \lrcorner).
\end{equation*}
Under the identification of $\bigwedge T^*X_\mathbb{C}$ with the Clifford algebra bundle over $(X, g)$, we know that $\omega_I, \omega_J, \omega_K$ generates a Lie subalgebra of $\Omega^2(X, \mathbb{C})$ which is isomorphic to $\mathfrak{sp}(1)$. Then the triple $(c_\zeta(\omega_I), c_\zeta(\omega_J), c_\zeta(\omega_K))$ defines an $\mathfrak{sp}(1)$-action on $\Omega^*(X, E)$. Indeed, this copy of $\mathfrak{sp}(1)$ lies in a copy of $\mathfrak{sl}(2, \mathbb{C})$ generated by a Lefschetz triple as introduced in Subsection \ref{Subsection 2.2}. To see this, take $J_\zeta = J$ and $\Omega = \tfrac{1}{2} (\omega_K + \sqrt{-1}\omega_I)$. We can check that
\begin{align*}
c_\mathbf{j}(\omega_J) = 2\sqrt{-1}H_{\overline{\Omega}}, \quad c_\mathbf{j}(\overline{\Omega}) = 2L_{\overline{\Omega}}, \quad \text{and} \quad c_\mathbf{j}(\Omega) = -2\Lambda_{\overline{\Omega}}.
\end{align*}
Therefore,
\begin{align*}
c_\mathbf{j}(\omega_J) = 2\sqrt{-1} H_{\overline{\Omega}}, \quad c_\mathbf{j}(\omega_K) = 2 (L_{\overline{\Omega}} - \Lambda_{\overline{\Omega}}), \quad \text{and} \quad c_\mathbf{j}(\omega_I) = 2\sqrt{-1} (L_{\overline{\Omega}} + \Lambda_{\overline{\Omega}}).
\end{align*}
We conclude the above observation in the following proposition.

\begin{proposition}
	\label{Proposition 2.5}
	The $\mathfrak{sp}(1)$-representation on $\Omega^*(X, E)$ generated by $c_\mathbf{j}(\omega_I), c_\mathbf{j}(\omega_J), c_\mathbf{j}(\omega_K)$ integrates to an $\operatorname{Sp}(1)$-representation $\rho_\mathbf{j}^{\operatorname{Sp}(1)}$ given by
	\begin{equation*}
		\rho_\mathbf{j}^{\operatorname{Sp}(1)} = \rho_\mathbf{j}^{\operatorname{SL}(2, \mathbb{C})} \circ \rho_1^{\operatorname{Sp}(1)},
	\end{equation*}
	where $\rho_\mathbf{j}^{\operatorname{SL}(2, \mathbb{C})}$ is the $\operatorname{SL}(2, \mathbb{C})$-representation given as in Proposition \ref{Proposition 2.4} and $\rho_1^{\operatorname{Sp}(1)}$ is the embedding $\operatorname{Sp}(1) \hookrightarrow \operatorname{SL}(2, \mathbb{C})$ given as in (\ref{Equation B.2}). In particular,
	\begin{align*}
	\rho_\mathbf{j}^{\operatorname{Sp}(1)}(\mathbf{j}) = e^{\tfrac{\pi}{2} \sqrt{-1} H_{\overline{\Omega}}}, \quad \rho_\mathbf{j}^{\operatorname{Sp}(1)}(\mathbf{k}) = e^{\frac{\pi}{2} ( L_{\overline{\Omega}} - \Lambda_{\overline{\Omega}} )}, \quad \text{and} \quad \rho_\mathbf{j}^{\operatorname{Sp}(1)}(\mathbf{i}) = e^{\frac{\pi}{2} \sqrt{-1} ( L_{\overline{\Omega}} + \Lambda_{\overline{\Omega}} )}.
	\end{align*}
\end{proposition}

On the other hand, we shall see how the $\operatorname{Sp}(1)$-action $\rho^{\operatorname{Sp}(1)}$ induced by the hypercomplex structure on $X$ intertwines the above family of fibrewise Clifford actions. The following proposition is useful for Theorem \ref{Theorem 1.1}. Its proof is by direct calculation.

\begin{proposition}
	\label{Proposition 2.6}
	Let $\rho := \rho^{\operatorname{Sp}(1)}$ be the action given as in (\ref{Equation 2.2}). Let $\zeta \in \mathbb{S}^2$ and $\eta \in \operatorname{Sp}(1)$. Then the following diagram commutes:
	\begin{center}
		\begin{tikzcd}
			T^*X_\mathbb{C} \ar[r, "c_\zeta"] \ar[d, "\rho(\eta)"'] & \operatorname{End}(\bigwedge T^*X_\mathbb{C} \otimes E) \ar[d, "\text{conjugation by } \rho(\eta)"] \\
			T^*X_\mathbb{C} \ar[r, "c_{\eta \cdot \zeta}"'] & \operatorname{End}(\bigwedge T^*X_\mathbb{C} \otimes E)
		\end{tikzcd}
	\end{center}
\end{proposition}

\section{$\operatorname{Spin}^{\operatorname{c}}$-Dirac Operators on a HyperK\"ahler Manifold}
\label{Section 3}
This is the main section of the present paper. Throughout this section, let $E$ be a Hermitian vector bundle over $X$ with a unitary connection $\nabla$. We first recall the following theorem proved by Verbitsky which is parallel to Theorem \ref{Theorem 1.1} and serves as a motivation of this paper.

\begin{theorem}
	\label{Theorem 3.1}
	(Reformulation of Theorem 8.1 in \cite{Ver1996})
	Let $X$ be a hyperK\"ahler manifold and $E$ be a hyperholomorphic bundle over $X$. Then for all $\eta \in \operatorname{Sp}(1)$ and $\zeta \in \mathbb{S}^2$, the following diagram commutes:
	\begin{equation}
	\begin{tikzcd}
	\Omega_{J_\zeta}^{0, *}(X, E) \ar[r, "\Delta_{\overline{\partial}_{J_\zeta}}"] \ar[d, "\eta \cdot"'] & \Omega_{J_\zeta}^{0, *}(X, E) \ar[d, "\eta \cdot"]\\
	\Omega_{J_{\eta \cdot \zeta}}^{0, *}(X, E) \ar[r, "\Delta_{\overline{\partial}_{J_{\eta \cdot \zeta}}}"'] & \Omega_{J_{\eta \cdot \zeta}}^{0, *}(X, E)
	\end{tikzcd}
	\end{equation}
	where $\Delta_{\overline{\partial}_{J_\zeta}} := (\overline{\partial}_{J_\zeta} + \overline{\partial}_{J_\zeta}^*)^2$ is the Dolbeault Laplacian on $\Omega_{J_\zeta}^{0, *}(X, E)$ and
	\begin{align*}
		\rho^{\operatorname{Sp}(1)}: \operatorname{Sp}(1) \times \Omega^*(X, E) \to \Omega^*(X, E), \quad (\eta, s) \mapsto \eta \cdot s,
	\end{align*}
	is the $\operatorname{Sp}(1)$-action on $\Omega^*(X, E)$ given as in (\ref{Equation 2.2}).
\end{theorem}

For an arbitrary Hermitian vector bundle $E$ over $X$ with a unitary connection $\nabla$, $E$ does not necessarily admit a $J_\zeta$-holomorphic structure, but we can always define the associated $\operatorname{Spin}^{\operatorname{c}}$-Dirac operator $\slashed{D}_{J_\zeta}$, which plays a key role in our main theorem. By abuse of notation, $\nabla$ also denotes the connection on $\bigwedge T^*X_\mathbb{C} \otimes E$ induced by $\nabla$ and the Levi-Civita connection $\nabla^g$ on $(X, g)$.

\begin{definition}
	\label{Definition 3.2}
	For $\zeta \in \mathbb{S}^2$, define the $\operatorname{Spin}^{\operatorname{c}}$-Dirac operator $\slashed{D}_{J_\zeta}$ on $\Omega^*(X, E)$ as the composition
	\begin{equation}
	\begin{tikzcd}
	\Omega^*(X, E) \ar[r, "\nabla"] & \Gamma(X, T^*X_\mathbb{C} \otimes \textstyle\bigwedge T^*X_\mathbb{C} \otimes E) \ar[r, "c_\zeta"] & \Omega^*(X, E),
	\end{tikzcd}
	\end{equation}
	where $c_\zeta$ is the fibrewise Clifford action given as in (\ref{Equation 2.5}).
\end{definition}

We can give an alternative description of $\slashed{D}_{J_\zeta}$. Denote by $d$ the exterior covariant derivative on $\Omega^*(X, E)$ induced by $\nabla$. In general, $d^2 \neq 0$. Consider any $\zeta \in \mathbb{S}^2$. Define $\overline{\partial}_{J_\zeta} = \tfrac{1}{2} ( d - \sqrt{-1} d^{J_\zeta} )$ on $\Omega^*(X, E)$, where $\Phi^{J_\zeta} = \rho(\zeta) \circ \Phi \circ \rho(\zeta^{-1})$ for any $\mathbb{C}$-linear map $\Phi: \Omega^*(X, E) \to \Omega^*(X, E)$ and $\rho = \rho^{\operatorname{Sp}(1)}$ is given as in (\ref{Equation 2.2}). When restricted on $\Omega^0(X, E)$, $\overline{\partial}_{J_\zeta}$ is the $(0, 1)$-part of the connection $\nabla$ with respect to the complex structure $J_\zeta$. Indeed, we have
\begin{equation}
\slashed{D}_{J_\zeta} = \sqrt{2}(\overline{\partial}_{J_\zeta} + \overline{\partial}_{J_\zeta}^*),
\end{equation}
and hence for any $p \in \mathbb{N}$, $\slashed{D}_{J_\zeta}$ preserves the space $\Omega_{J_\zeta}^{p, *}(X, E)$, i.e. $\slashed{D}_{J_\zeta}(\Omega_{J_\zeta}^{p, *}(X, E)) \subset \Omega_{J_\zeta}^{p, *}(X, E)$.\par
Now, we are ready to prove our main result. In Subsection \ref{Subsection 3.1}, we provide the proofs of Theorems \ref{Theorem 1.1} and Corollary \ref{Corollary 1.2}. In Subsection \ref{Subsection 3.2}, we further discuss a relation between $\slashed{D}_J$ and $\slashed{D}_{-J}$ for an arbitrary unitary connection $\nabla$ on a Hermitian vector bundle $E$ over $X$.

\subsection{The proof of Theorem \ref{Theorem 1.1}}
\quad\par
\label{Subsection 3.1}
Our main result concerns with an $\operatorname{Sp}(1)$-symmetry intertwining an $\mathbb{S}^2$-indexed family $\{ \slashed{D}_{J_\zeta}^2 \}_{\zeta \in \mathbb{S}^2}$ of the $\operatorname{Spin}^{\operatorname{c}}$-Dirac Laplacians when the curvature $F_E$ of $\nabla$ satisfies $\tfrac{\sqrt{-1}}{2\pi}F_E = \omega_J \operatorname{Id}_E$.

\begin{theorem}
	($=$ Theorem \ref{Theorem 1.1})
	Let $X$ be a hyperK\"ahler manifold and $E$ be a Hermitian vector bundle over $X$ with a unitary connection of curvature $F_E$ with $\tfrac{\sqrt{-1}}{2\pi}F_E = \omega_J \operatorname{Id}_E$. Then there is an $\operatorname{Sp}(1)$-action
	\begin{align*}
	\chi: \operatorname{Sp}(1) \times \Omega^*(X, E) \to \Omega^*(X, E), \quad (\eta, s) \mapsto \eta \cdot s,
	\end{align*}
	such that for all $\eta \in \operatorname{Sp}(1)$ and $\zeta \in \mathbb{S}^2$, the following diagram commutes:
	\begin{center}
		\begin{tikzcd}
			\Omega_{J_\zeta}^{0, *}(X, E) \ar[d, "\eta \cdot"'] \ar[r, "\slashed{D}_{J_\zeta}^2"] & \Omega_{J_\zeta}^{0, *}(X, E) \ar[d, "\eta \cdot"]\\
			\Omega_{J_{\eta \cdot \zeta}}^{0, *}(X, E) \ar[r, "\slashed{D}_{J_{\eta \cdot \zeta}}^2"'] & \Omega_{J_{\eta \cdot \zeta}}^{0, *}(X, E)
		\end{tikzcd}
	\end{center}
	where $\slashed{D}_{J_\zeta}$ is the $\operatorname{Spin}^{\operatorname{c}}$-Dirac operator on $\Omega_{J_\zeta}^{0, *}(X, E)$ defined as in Definition \ref{Definition 3.2}.
\end{theorem}

A naive guess of the desired $\operatorname{Sp}(1)$-symmetry for Theorem \ref{Theorem 1.1} would be the $\operatorname{Sp}(1)$-action $\rho^{\operatorname{Sp}(1)}$ in (\ref{Equation 2.2}), because it is the action that appears in Theorem \ref{Theorem 3.1}. It turns out that this is not true. Yet, we find that this `wrong' $\operatorname{Sp}(1)$-action $\rho^{\operatorname{Sp}(1)}$ is useful in the proof of Theorem \ref{Theorem 1.1} - it transforms the $\mathbb{S}^2$-indexed family $\{\slashed{D}_{J_\zeta}^2\}_{\zeta \in \mathbb{S}^2}$ of second order differential operators to another $\mathbb{S}^2$-indexed family $\{ \tilde{\Delta}_\zeta \}_{\zeta \in \mathbb{S}^2}$ of operators, and as a result it is much easier to write down an $\operatorname{Sp}(1)$-symmetry intertwining the operators $\tilde{\Delta}_\zeta$'s and preserving the spaces $\Omega_J^{0, *}(X, E)$.

\begin{definition}
	\label{Definition 3.4}
	Let $\zeta \in \mathbb{S}^2$. Define the differential operator $\tilde{\Delta}_\zeta$ on $\Omega_J^{0, *}(X, E)$ as
	\begin{equation}
		\tilde{\Delta}_\zeta = \nabla^*\nabla - 2 \pi \sqrt{-1} c_\mathbf{j} ( \omega_\zeta ).
	\end{equation}
\end{definition}

\begin{lemma}
	\label{Lemma 3.5}
	Let $\rho := \rho^{\operatorname{Sp}(1)}$ be the action given as in (\ref{Equation 2.2}). Let $\zeta \in \mathbb{S}^2 \subset \mathfrak{sp}(1)$ and $\alpha \in \operatorname{Sp}(1)$ be such that $\alpha \cdot \zeta = \mathbf{j}$. Then the following diagram commutes:
	\begin{center}
		\begin{tikzcd}
			\Omega_{J_\zeta}^{0, *}(X, E) \ar[d, "\rho(\alpha)"'] \ar[r, "\slashed{D}_{J_\zeta}^2"] & \Omega_{J_\zeta}^{0, *}(X, E) \ar[d, "\rho(\alpha)"]\\
			\Omega_J^{0, *}(X, E) \ar[r, "\tilde{\Delta}_\zeta"'] & \Omega_J^{0, *}(X, E)
		\end{tikzcd}
	\end{center}
	where $\slashed{D}_{J_\zeta}$ is the Dirac operator on $\Omega_{J_\zeta}^{0, *}(X, E)$ defined as in Definition \ref{Definition 3.2} and $\tilde{\Delta}_\zeta$ is defined as in Definition \ref{Definition 3.4}.
\end{lemma}
To prove Lemma \ref{Lemma 3.5}, we need another useful lemma.

\begin{lemma}
	\label{Lemma 3.6}
	Let $E$ be a Hermitian vector bundle over $X$ with a unitary connection $\nabla$.  The covariant Laplacian $\nabla^*\nabla$ commutes with the action $\rho^{\operatorname{Sp}(4, \mathbb{C})}$ on $\Omega^*(X, E)$ as in Proposition \ref{Proposition 2.1}.
\end{lemma}
\begin{proof}
	Fix $\zeta \in \mathbb{S}^2$. We will see that it suffices to prove that $\nabla^*\nabla$ commutes with $L_{\omega_\zeta}$. Recall that $\nabla^*\nabla$ can be computed as follows. Let
	\begin{align*}
	\operatorname{tr}_g: \Gamma(X, T^*X \otimes T^*X) \to \mathcal{C}^\infty(X)
	\end{align*}
	be the composition of the isomorphism $\Gamma(X, T^*X \otimes T^*X) \overset{\cong}{\to} \Gamma(X, \operatorname{End}(TX))$ induced by the hyperK\"ahler metric $g$ of $X$ and the trace map. Let $E_0 = \bigwedge T^*X_\mathbb{C} \otimes E$ and $\nabla^{(2)}: \Gamma(X, E_0) \to \Gamma(X, T^*X \otimes T^*X \otimes E_0)$ be the second order covariant derivative associated to $\nabla$:
	\begin{align*}
	\nabla_{u, v}^{(2)}	s = \nabla_u\nabla_vs - \nabla_{\nabla_u^gv}s,
	\end{align*}
	for all $u, v \in \Gamma(X, TX)$ and $s \in \Gamma(X, E_0)$. Then $\nabla^*\nabla = -(\operatorname{tr}_g \otimes \operatorname{Id}_{E_0}) \circ \nabla^{(2)}$.\par
	Now consider any $s \in \Gamma(X, E_0)$ and $u, v \in \Gamma(X, TX)$. As $X$ is hyperK\"ahler, $\nabla^g \omega_\zeta = 0$. Then
	\begin{align*}
	\nabla_uL_{\omega_\zeta}s = \nabla_u(\omega_\zeta \wedge s) = \omega_\zeta \wedge \nabla_us = L_{\omega_\zeta}\nabla_us,
	\end{align*}
	whence $\nabla_{u, v}^{(2)}L_{\omega_\zeta}s = L_{\omega_\zeta} \nabla_{u, v}^{(2)}s$. Therefore, $[\nabla^*\nabla, L_{\omega_\zeta}] = 0$.\par
	Next, taking Hermitian adjoints, we see that $[\nabla^*\nabla, \Lambda_{\omega_\zeta}] = 0$. As $H = [L_{\omega_\zeta}, \Lambda_{\omega_\zeta}]$, $[\nabla^*\nabla, H] = 0$. Finally, note that $\operatorname{ad}(K) = [L_{\omega_I}, \Lambda_{\omega_J}]$. By the Jacobi identity, $[\nabla^*\nabla, \operatorname{ad}(K)] = 0$, and similarly $[\nabla^*\nabla, \operatorname{ad}(I)] = [\nabla^*\nabla, \operatorname{ad}(J)] = 0$. Hence, $\nabla^*\nabla$ commutes with $\rho^{\mathfrak{sp}(4, \mathbb{C})}$ and  $\rho^{\operatorname{Sp}(4, \mathbb{C})}$ as well.
\end{proof}

Similar to the operator $\tilde{\Delta}_\zeta$, $\slashed{D}_{J_\zeta}^2$ differs from $\nabla^*\nabla$ by a zeroth order term.

\begin{proof}[\myproof{Lemma}{\ref{Lemma 3.5}}]
	Since $X$ is hyperK\"ahler, it is Ricci-flat and hence has zero scalar curvature. By the Lichnerowicz formula,
	\begin{align*}
	\slashed{D}_{J_\zeta}^2 = \nabla^*\nabla - 2 \pi \sqrt{-1} c_\zeta(\omega_J).
	\end{align*}
	By Lemma \ref{Lemma 3.6}, $\nabla^*\nabla$ is invariant under the $\operatorname{Sp}(1)$-action $\rho$, whence $\rho(\alpha) \circ \nabla^*\nabla = \nabla^*\nabla \circ \rho(\alpha)$. Then by Proposition \ref{Proposition 2.6}, we can see that $\rho(\alpha)$ transforms the space $\Omega_{J_\zeta}^{0, *}(X, E)$ to $\Omega_J^{0, *}(X, E)$ and the operator $\slashed{D}_{J_\zeta}^2$ to $\tilde{\Delta}_\zeta$.
\end{proof}

The following lemma shows that there is an $\operatorname{Sp}(1)$-symmetry inherited from the $\mathfrak{sl}(2, \mathbb{C})$-action generated by the Lefschetz triple $(L_{\overline{\Omega}}, \Lambda_{\overline{\Omega}}, H_{\overline{\Omega}})$, where $\Omega = \tfrac{1}{2}(\omega_K + \sqrt{-1}\omega_I)$, intertwining $\tilde{\Delta}_\zeta$'s.

\begin{lemma}
	\label{Lemma 3.7}
	The $\operatorname{Sp}(1)$-action $\rho := \rho_\mathbf{j}^{\operatorname{Sp}(1)}$ on $\Omega_J^{0, *}(X, E)$ given as in Proposition \ref{Proposition 2.5} satisfies the condition that for all $\eta \in \operatorname{Sp}(1)$ and $\zeta \in \mathbb{S}^2$, the following diagram commutes:
	\begin{center}
		\begin{tikzcd}
			\Omega_J^{0, *}(X, E) \ar[d, "\rho(\eta)"'] \ar[r, "\tilde{\Delta}_\zeta"] & \Omega_J^{0, *}(X, E) \ar[d, "\rho(\eta)"]\\
			\Omega_J^{0, *}(X, E) \ar[r, "\tilde{\Delta}_{\eta \cdot \zeta}"'] & \Omega_J^{0, *}(X, E)
		\end{tikzcd}
	\end{center}
	where $\tilde{\Delta}_\zeta$ is defined as in Definition \ref{Definition 3.4}.
\end{lemma}
\begin{proof}
	Consider any $\eta \in \operatorname{Sp}(1)$ and $\zeta = \zeta_I\mathbf{i} + \zeta_J\mathbf{j} + \zeta_K \mathbf{k} \in \mathbb{S}^2 \subset \mathfrak{sp}(1)$. Note that
	\begin{align*}
	\tfrac{1}{2} c_\mathbf{j}(\omega_\zeta) = \zeta_J \cdot \sqrt{-1} H_{\overline{\Omega}} + \zeta_K \cdot (L_{\overline{\Omega}} - \Lambda_{\overline{\Omega}}) + \zeta_I \cdot \sqrt{-1} (L_{\overline{\Omega}} + \Lambda_{\overline{\Omega}}).
	\end{align*}
	It follows from Proposition \ref{Proposition 2.5} that $\rho(\eta) \circ c_\mathbf{j}(\omega_\zeta) = c_\mathbf{j}(\omega_{\eta \cdot \zeta}) \circ \rho(\eta)$. On the other hand, by Lemma \ref{Lemma 3.6}, we have $\rho(\eta) \circ \nabla^*\nabla = \nabla^*\nabla \circ \rho(\eta)$. We are done.
\end{proof}

We shall finish the proof of Theorem \ref{Theorem 1.1}.

\begin{proof}[\myproof{Theorem}{\ref{Theorem 1.1}}]
	By transitivity of the $\operatorname{Sp}(1)$-action on $\mathbb{S}^2 \subset \mathfrak{sp}(1)$, for each $\zeta \in \mathbb{S}^2$ we can pick $\alpha_\zeta \in \operatorname{Sp}(1)$ such that $\alpha_\zeta \cdot \zeta = \mathbf{j}$. Then for all $\eta \in \operatorname{Sp}(1)$, define $\chi(\eta)$ as the composition
	\begin{center}
		\begin{tikzcd}
			\Omega_{J_\zeta}^{*, *}(X, E) \ar[r, "\rho(\alpha_\zeta)"] & \Omega_J^{*, *}(X, E) \ar[r, "\rho_\mathbf{j}(\eta)"] & \Omega_J^{*, *}(X, E) \ar[rr, "(\rho(\alpha_{\eta \cdot \zeta}))^{-1}"] && \Omega_{J_{\eta \cdot \zeta}}^{*, *}(X, E)
		\end{tikzcd}
	\end{center}
	where $\rho := \rho^{\operatorname{Sp}(1)}$ and $\rho_\mathbf{j} = \rho_\mathbf{j}^{\operatorname{Sp}(1)}$. Then we can easily see that $\chi: \operatorname{Sp}(1) \to \operatorname{End}(\Omega^*(X, E))$ is an action and for all $\eta \in \operatorname{Sp}(1)$ and $\zeta \in \mathbb{S}^2$, $\chi(\eta)(\Omega_{J_\zeta}^{0, *}(X, E)) = \Omega_{J_{\eta \cdot \zeta}}^{0, *}(X, E)$. Eventually, by Lemmas \ref{Lemma 3.5} and \ref{Lemma 3.7}, for all $\eta \in \mathbb{S}^2$ and $\eta \in \operatorname{Sp}(1)$, we have
	\begin{align*}
		\chi(\eta) \circ \slashed{D}_{J_\zeta}^2 = \slashed{D}_{J_{\eta \cdot \zeta}}^2 \circ \chi(\eta).
	\end{align*}
\end{proof}

\begin{remark}
	\label{Remark 3.8}
	Our approach in the proof of Theorem \ref{Theorem 1.1} can also be used to reprove Theorem \ref{Theorem 3.1}. We give a sketch as follows. Let $E$ be a hyperholomorphic bundle over $X$. Again, we can apply the Lichnerowicz formula to express $\slashed{D}_{J_\zeta}^2$ in terms of the covariant Laplacian $\nabla^*\nabla$ and its Weitzenb\"ock remainder. By Lemma \ref{Lemma 3.6}, $\nabla^*\nabla$ is invariant under the $\operatorname{Sp}(1)$-action $\rho^{\operatorname{Sp}(1)}$ given as in (\ref{Equation 2.2}). The Weitzenb\"ock remainder of $\slashed{D}_{J_\zeta}^2$ involves the curvature $F_E$, which is also invariant under $\rho^{\operatorname{Sp}(1)}$ by Proposition 1.2 in \cite{Ver1996}. Then by Proposition \ref{Proposition 2.6}, $\rho^{\operatorname{Sp}(1)}$ intertwines the operators $\slashed{D}_{J_\zeta}^2$.
\end{remark}

What follows from Theorem \ref{Theorem 1.1} is the following vanishing theorem.

\begin{corollary}
	($=$ Corollary \ref{Corollary 1.2})
	With the same assumption as in Theorem \ref{Theorem 1.1}, for all $\zeta \in \mathbb{S}^2$, we have
	\begin{align*}
	\ker \slashed{D}_{J_\zeta}^- = 0,
	\end{align*}
	where $\slashed{D}_{J_\zeta}^-$ is the restriction of $\slashed{D}_{J_\zeta}$ on the odd degree component $\Omega_{J_\zeta}^{0, 2*+1}(X, E)$.
\end{corollary}
\begin{proof}
	By Akizuki-Nakano's proof \cite{AkiNak1954} of Kodaira vanishing theorem, the kernel of the restriction of $\slashed{D}_J^2$ on $\Omega_J^{2n, q}(X, E)$ is zero for $q > 0$. Lemma \ref{Lemma 3.14}, to be proved in Subsection \ref{Subsection 3.2}, states that $e^{-\frac{\pi}{2}(L_{\Omega} - \Lambda_{\Omega})}: \Omega_J^{0, q}(X, E) \to \Omega_J^{2n, q}(X, E)$ commutes with $\slashed{D}_J^2$, where $\Omega = \tfrac{1}{2}(\omega_K + \sqrt{-1}\omega_I)$. Thus, $\ker \slashed{D}_J^2 \cap \Omega_J^{0, q}(X, E) = 0$.\par
	Now fix $\zeta \in \mathbb{S}^2$. Since the $\operatorname{Sp}(1)$-action appeared in Theorem \ref{Theorem 1.1} preserves $\mathbb{Z}_2$-grading, the kernel of the restriction of $\slashed{D}_{J_\zeta}^2$ on $\Omega_{J_\zeta}^{0, 2*+1}(X, E)$ is zero. As $\ker \slashed{D}_{J_\zeta} = \ker \slashed{D}_{J_\zeta}^2$, we have $\ker \slashed{D}_{J_\zeta}^- = 0$.
\end{proof}

\subsection{Hodge star operator intertwining $\slashed{D}_J$ and $\slashed{D}_{-J}$}
\quad\par
\label{Subsection 3.2}
An observation from Theorem \ref{Theorem 1.1} is that there is an $\mathbb{S}^1$-indexed family of isomorphisms
\begin{align*}
\chi(\mu): \Omega_J^{0, *}(X, E) \to \Omega_{-J}^{0, *}(X, E)
\end{align*}
parametrized by elements $\mu \in \operatorname{Sp}(1)$ with $\mu \cdot \mathbf{j} = -\mathbf{j}$, i.e. unit quaternions $\mu \in \mathfrak{sp}(1)$ orthogonal to $\mathbf{j}$, such that $\chi(\mu) \circ \slashed{D}_J^2 = \slashed{D}_{-J}^2 \circ \chi(\mu)$. In this subsection, we shall show that even for an arbitrary Hermitian vector bundle $E$ over $X$ with a unitary connection $\nabla$, there still exists an $\mathbb{S}^1$-indexed family of isomorphisms naturally parametrized by such $\mu$'s intertwining $\slashed{D}_J$ and $\slashed{D}_{-J}$, and these isomorphisms are induced by the Hodge star operator on the hyperK\"ahler manifold $X$.

\begin{theorem}
	\label{Theorem 3.10}
	Let $X$ be a hyperK\"ahler manifold and $E$ be a Hermitian vector bundle over $X$ with a unitary connection $\nabla$. Then for all $p \in \mathbb{N}$, the following diagram commutes:
	\begin{center}
		\begin{tikzcd}
			\Omega_J^{p, *}(X, E) \ar[d, "\chi(\mathbf{k})"'] \ar[r, "\slashed{D}_J"] & \Omega_J^{p, *}(X, E) \ar[d, "\chi(\mathbf{k})"]\\
			\Omega_{-J}^{p, *}(X, E) \ar[r, "\slashed{D}_{-J}"'] & \Omega_{-J}^{p, *}(X, E)
		\end{tikzcd}
	\end{center}
	where $\chi(\mathbf{k}) := \rho^{\operatorname{Sp}(1)}(\mathbf{k}) \circ \rho_\mathbf{j}^{\operatorname{Sp}(1)}(\mathbf{k})$, and $\rho^{\operatorname{Sp}(1)}$, $\rho_\mathbf{j}^{\operatorname{Sp}(1)}$ are the $\operatorname{Sp}(1)$-actions given as in (\ref{Equation 2.2}) and in Proposition \ref{Proposition 2.5} respectively.
\end{theorem}

Clearly, this theorem still holds when $\mathbf{k}$ is replaced by any $\mu \in \operatorname{Sp}(1)$ such that $\mu \cdot \mathbf{j} = -\mathbf{j}$.\par
Write $\rho := \rho^{\operatorname{Sp}(1)}$ and $\rho_\mathbf{j} := \rho_\mathbf{j}^{\operatorname{Sp}(1)}$. Since $\chi(\mathbf{k})$ is written as the composition of two actions $\Omega_J^{p, q}(X, E) \overset{\rho_\mathbf{j}(\mathbf{k})}{\longrightarrow} \Omega_J^{p, 2n-q}(X, E) \overset{\rho(\mathbf{k})}{\longrightarrow} \Omega_{-J}^{p, 2n-q}(X, E)$, one might expect to prove Theorem \ref{Theorem 3.10} by claiming that $\rho_\mathbf{j}(\mathbf{k})$ commutes with $\slashed{D}_J$ and $\rho(\mathbf{k})$ intertwines $\slashed{D}_J$ and $\slashed{D}_{-J}$. Nonetheless, neither the former part nor the latter part of this claim is true in general. Instead, we shall express $\chi(\mathbf{k})$ as the composition of other two operators. To realize it, we need to consider another operator $*$ acting on $\Omega^*(X, E)$, which is essentially the Hodge star operator $\star$ for $(X, g)$ but up to signs on different degree components of $\Omega^*(X, E)$. It is given as follows. For $k \in \mathbb{N}$ and $s \in \Omega^k(X, E)$,
\begin{equation}
\label{Equation 3.5}
* s = (-1)^{\frac{1}{2}k(k+1)} \star s.
\end{equation}
The geometric meaning of $*$ comes from generalized K\"ahler geometry - $*$ serves as a generalization of $\star$. Interested readers are referred to \cite{Cav2006, Gua2004}.\par
Now we have the following lemma.

\begin{lemma}
	\label{Lemma 3.11}
	For all $p, q \in \mathbb{N}$, the following diagram commutes:
	\begin{equation}
	\label{Equation 3.6}
	\begin{tikzcd}
	\Omega_J^{p, q}(X, E) \ar[d, "\rho_\mathbf{j}(\mathbf{k})"'] \ar[r, "*"] & \Omega_J^{2n-q, 2n-p}(X, E)\\
	\Omega_J^{p, 2n-q}(X, E) \ar[r, "\rho(\mathbf{k})"'] & \Omega_J^{2n-q, p}(X, E) \ar[u, "\rho_\mathbf{j}(\mathbf{k})"']
	\end{tikzcd}
	\end{equation}
	In other words, letting $\Omega = \tfrac{1}{2}(\omega_K + \sqrt{-1}\omega_I)$, we have
	\begin{align*}
	* = e^{\frac{\pi}{2} ( L_{\overline{\Omega}} - \Lambda_{\overline{\Omega}} )} \circ \rho(\mathbf{k}) \circ e^{\frac{\pi}{2} ( L_{\overline{\Omega}} - \Lambda_{\overline{\Omega}} )}.
	\end{align*}
\end{lemma}
\begin{proof}
	We first check whether the maps appeared in Diagram (\ref{Equation 3.6}) transform the bi-degree on $\Omega_J^{*, *}(X, E)$ in the way as shown. It is evident that $*(\Omega_J^{p, q}(X, E)) = \Omega_J^{2n-q, 2n-p}(X, E)$ and $\rho(\mathbf{k})(\Omega_J^{p, q}(X, E)) = \Omega_J^{q, p}(X, E)$. As seen from Proposition \ref{Proposition 2.5}, $\rho_\mathbf{j}(\mathbf{k}) = e^{\frac{\pi}{2} ( L_{\overline{\Omega}} - \Lambda_{\overline{\Omega}} )}$. By Proposition \ref{Proposition 2.4} and (\ref{Equation B.4}), we can check that $\rho_\mathbf{j}(\mathbf{k})(\Omega_J^{p, q}(X, E)) = \Omega_J^{p, 2n-q}(X, E)$.\par
	Now the key point in our proof is a result in generalized K\"ahler geometry (Lemma 1.6 in \cite{Cav2012}):
	\begin{equation}
	\label{Equation 3.7}
	* = e^{\frac{\pi}{2}(L_{\omega_K} - \Lambda_{\omega_K})} \circ \rho(\mathbf{k}).
	\end{equation}
	Indeed, (\ref{Equation 3.7}) is a reformulation of Weil's formula \cite{Wei1958} (c.f. Formula (2) in \cite{CatTarTom2022}). As we can easily see that $\rho(\mathbf{k}) \circ e^{\frac{\pi}{2} ( L_{\overline{\Omega}} - \Lambda_{\overline{\Omega}} )} = e^{\frac{\pi}{2} ( L_{\Omega} - \Lambda_{\Omega} )} \circ \rho(\mathbf{k})$, we eventually have
	\begin{align*}
	e^{\frac{\pi}{2} ( L_{\overline{\Omega}} - \Lambda_{\overline{\Omega}} )} \circ \rho(\mathbf{k}) \circ e^{\frac{\pi}{2} ( L_{\overline{\Omega}} - \Lambda_{\overline{\Omega}} )} = e^{\frac{\pi}{2} ( L_{\overline{\Omega}} - \Lambda_{\overline{\Omega}} )} \circ e^{\frac{\pi}{2} ( L_{\Omega} - \Lambda_{\Omega} )} \circ \rho(\mathbf{k}) = e^{\frac{\pi}{2}(L_{\omega_K} - \Lambda_{\omega_K})} \circ \rho(\mathbf{k}) = *.
	\end{align*}
\end{proof}

By the above lemma, we can rewrite $\chi(\mathbf{k})$ as
\begin{align*}
	\chi(\mathbf{k}) = \rho_\mathbf{j}(-\mathbf{k}) \circ * = e^{-\frac{\pi}{2}(L_{\overline{\Omega}} - \Lambda_{\overline{\Omega}})} \circ *.
\end{align*}
Then the proof of Theorem \ref{Theorem 3.10} breaks into two steps. The first step is to show that $*$ intertwines $\slashed{D}_J$ and $\slashed{D}_{-J}$. For simplicity, we write $\overline{\partial} := \overline{\partial}_J$, $\partial := \overline{\partial}_{-J}$.

\begin{lemma}
	\label{Lemma 3.12}
	For all $p \in \mathbb{N}$, the following diagram commutes:
	\begin{center}
		\begin{tikzcd}
			\Omega_J^{p, *}(X, E) \ar[d, "*"'] \ar[r, "\slashed{D}_J"] & \Omega_J^{p, *}(X, E) \ar[d, "*"]\\
			\Omega_{-J}^{2n-p, *}(X, E) \ar[r, "\slashed{D}_{-J}"'] & \Omega_{-J}^{2n-p, *}(X, E)
		\end{tikzcd}
	\end{center}
\end{lemma}
\begin{proof}
	By (\ref{Equation 3.5}), it is easy to check that $\partial^* = * \overline{\partial} *^{-1}$ and $\partial = * \overline{\partial}^* *^{-1}$. Then we are done as
	\begin{align*}
		\slashed{D}_J = \sqrt{2} ( \overline{\partial} + \overline{\partial}^* ) \quad \text{and} \quad \slashed{D}_{-J} = \sqrt{2} ( \partial + \partial^* ).
	\end{align*}
\end{proof}

The second step is to show that $e^{-\frac{\pi}{2}(L_{\overline{\Omega}} - \Lambda_{\overline{\Omega}})}$ commutes with $\slashed{D}_J$. We need a further lemma so as to prove this statement.

\begin{lemma}
	\label{Lemma 3.13}
	We have $[L_{\overline{\Omega}}, \slashed{D}_{-J}] = [\Lambda_{\overline{\Omega}}, \slashed{D}_{-J}] = 0$ on $\Omega^*(X, E)$.
\end{lemma}
\begin{proof}
	Recall that $\slashed{D}_{-J}^* = \sqrt{2}(\partial + \partial^*)$. Since $\slashed{D}_{-J}^* = \slashed{D}_{-J}$, $[L_{\overline{\Omega}}, \slashed{D}_{-J}] = 0$ implies $[\Lambda_{\overline{\Omega}}, \slashed{D}_{-J}] = 0$. Thus, it suffices to prove the claim that $[L_{\overline{\Omega}}, \partial] = 0$ and $[L_{\overline{\Omega}}, \partial^*] = 0$. Indeed, the key identity is
	\begin{align*}
		[L_{\omega_\zeta}, d^*] = d^{J_\zeta}, \quad \text{for all } \zeta \in \mathbb{S}^2,
	\end{align*}
	This identity holds even when $E$ is not $J_\zeta$-holomorphic. Following similar ideas as Verbitsky's work on quaternionic Dolbeault complex \cite{Ver1996}, we prove the above claim as follows.
	\begin{align*}
	[L_{\omega_K}, \partial^*] = & \tfrac{1}{2} [L_{\omega_K}, d^* - \sqrt{-1} (d^*)^J] = \tfrac{1}{2} (d^K - \sqrt{-1} [(L_{\omega_K})^{J^{-1}}, d^*]^J)\\
	= & \tfrac{1}{2} (d^K + \sqrt{-1} [L_{\omega_K}, d^*]^J) = \tfrac{1}{2} (d^K + \sqrt{-1} (d^K)^J) = \tfrac{1}{2} (d^K + \sqrt{-1} d^I).
	\end{align*}
	Similarly, $[L_{\omega_I}, \partial^*] = \tfrac{1}{2} (d^I - \sqrt{-1} d^K)$. Therefore, $[L_{\overline{\Omega}}, \partial^*] = \tfrac{1}{2}[L_{\omega_K} - \sqrt{-1} L_{\omega_I}, \partial^*] = 0$. On the other hand, since $\overline{\Omega}$ is $\partial$-closed, $[L_{\overline{\Omega}}, \partial] = 0$ on $\Omega^*(X, E)$.
\end{proof}

Then we finish our second step by proving the following lemma.

\begin{lemma}
	\label{Lemma 3.14}
	For all $p \in \mathbb{N}$, the following diagram commutes:
	\begin{center}
		\begin{tikzcd}
			\Omega_{-J}^{2n-p, *}(X, E) \ar[d, "e^{-\frac{\pi}{2}(L_{\overline{\Omega}} - \Lambda_{\overline{\Omega}})}"'] \ar[r, "\slashed{D}_{-J}"] & \Omega_{-J}^{2n-p, *}(X, E) \ar[d, "e^{-\frac{\pi}{2}(L_{\overline{\Omega}} - \Lambda_{\overline{\Omega}})}"]\\
			\Omega_{-J}^{p, *}(X, E) \ar[r, "\slashed{D}_{-J}"'] & \Omega_{-J}^{p, *}(X, E)
		\end{tikzcd}
	\end{center}
\end{lemma}
\begin{proof}
	By Lemma \ref{Lemma 3.13}, $[L_{\overline{\Omega}}, \slashed{D}_{-J}] = [\Lambda_{\overline{\Omega}}, \slashed{D}_{-J}] = 0$. Then $\slashed{D}_{-J}$ is invariant under the $\mathfrak{u}(1)$-action generated by $L_{\overline{\Omega}} - \Lambda_{\overline{\Omega}}$ and thus invariant under the $\operatorname{U}(1)$-action integrating the $\mathfrak{u}(1)$-action. Thus,
	\begin{align*}
	e^{-\frac{\pi}{2}(L_{\overline{\Omega}} - \Lambda_{\overline{\Omega}})} \slashed{D}_{-J} = \slashed{D}_{-J} e^{-\frac{\pi}{2}(L_{\overline{\Omega}} - \Lambda_{\overline{\Omega}})}.
	\end{align*}
\end{proof}

With the above lemmas, the proof of Theorem \ref{Theorem 3.10} becomes straightforward.

\begin{proof}[\myproof{Theorem}{\ref{Theorem 3.10}}]
	We have
	\begin{align*}
		\chi(\mathbf{k}) \circ \slashed{D}_J = & e^{-\frac{\pi}{2}(L_{\overline{\Omega}} - \Lambda_{\overline{\Omega}})} \circ * \circ \slashed{D}_J = e^{-\frac{\pi}{2}(L_{\overline{\Omega}} - \Lambda_{\overline{\Omega}})} \circ \slashed{D}_{-J} \circ *\\
		= & \slashed{D}_{-J} \circ e^{-\frac{\pi}{2}(L_{\overline{\Omega}} - \Lambda_{\overline{\Omega}})} \circ * = \slashed{D}_{-J} \circ \chi(\mathbf{k}).
	\end{align*}
	The first equality and the last equality are due to Lemma \ref{Lemma 3.11}. The second equality follows from Lemma \ref{Lemma 3.12} and the third from Lemma \ref{Lemma 3.14}.
\end{proof}

\section{$\operatorname{Sp}(1)$-symmetries for Morphism Spaces between Hyperbranes}
\label{Section 4}
This last section is a discussion on the physical implications of our main result. As mentioned in Subsection \ref{Subsection 1.1}, Theorems \ref{Theorem 1.1} and \ref{Theorem 3.1} have physical implications on (A, B, A)-branes and (B, B, B)-branes respectively. We shall unify these two notions via generalized hypercomplex geometry. A review of the background framework is included in Appendix \ref{Appendix A} for readers who are unfamiliar with generalized geometry.

\begin{definition}
	\label{Definition 4.1}
	A (rank-$1$) \emph{hyperbrane} on a generalized hypercomplex manifold $(M, \mathcal{J})$ is a submanifold $S$ of $M$ together with a Hermitian line bundle $L$ with a unitary connection, called its \emph{Chan-Paton bundle}, such that $(S, L)$ is a brane on $(M, \mathcal{J}_\zeta)$ (see Definition \ref{Definition A.3}) for all $\zeta \in \mathbb{S}^2$, where $\{ \mathcal{J}_\zeta \}_{\zeta \in \mathbb{S}^2}$ is the twistor family of generalized complex structures on $M$.
\end{definition}

If a generalized hypercomplex structure $\mathcal{J}$ on a smooth manifold $M$ is induced by a hypercomplex structure as in Example \ref{Example A.7} (resp. a holomorphic symplectic structure as in Example \ref{Example A.8}), then we call hyperbranes on $(M, \mathcal{J})$ \emph{(B, B, B)-branes} (resp. \emph{(A, B, A)-branes}).\par
Throughout the rest of this section, unless otherwise specified, let $\mathcal{B}_0 = (S_0, L_0)$ and $\mathcal{B}_1 = (S_1, L_1)$ be hyperbranes on $X$, where $X$ is equipped with a generalized hyperK\"ahler structure arising from a hyperK\"ahler structure on it (as in Remark \ref{Remark A.12}), such that the intersection $S$ of $S_0$ and $S_1$ is clean and compact, and define $\widetilde{L} = \operatorname{Hom}(L_0 \vert_S, L_1 \vert_S)$. We expect the morphism spaces from $\mathcal{B}_0$ to $\mathcal{B}_1$ with respect to the $\mathbb{S}^2$-indexed family of generalized complex structures are intertwined by an $\operatorname{Sp}(1)$-symmetry. We shall discuss whether our expectation is verified in the following different situations.

\subsection{$\operatorname{Sp}(1)$-symmetries for morphism spaces between (B, B, B)-branes}
\quad\par
Suppose $\mathcal{B}_0, \mathcal{B}_1$ are (B, B, B)-branes on $X$. As an illustration, assume $\mathcal{B}_0, \mathcal{B}_1$ are space-filling. For each $\zeta \in \mathbb{S}^2$, the morphism space $\operatorname{Hom}_{J_\zeta}(\mathcal{B}_0, \mathcal{B}_1)$ with respect to the complex structure $J_\zeta$ in the twistor family can be computed as the $\overline{\partial}_{J_\zeta}$-cohomology of $\Omega_{J_\zeta}^{0, *}(X, \widetilde{L})$, which is isomorphic to the kernel of the Dolbeault Laplacian $\Delta_{\overline{\partial}_{J_\zeta}}$ by Hodge theory. Then we can apply Theorem \ref{Theorem 3.1} to establish an $\operatorname{Sp}(1)$-symmetry intertwining $\operatorname{Hom}_{J_\zeta}(\mathcal{B}_0, \mathcal{B}_1)$.\par
Now consider the general case when $\mathcal{B}_0, \mathcal{B}_1$ are not necessarily space-filling. We first explain how $\operatorname{Hom}_{J_\zeta}(\mathcal{B}_0, \mathcal{B}_1)$ is computed as a Dolbeault cohomology for a fixed $\zeta \in \mathbb{S}^2$. At the moment we refer `holomorphic' to `$J_\zeta$-holomorphic' and suppress the suffix $J_\zeta$ for relevant notations. In terms of derived category, $\operatorname{Hom}(\mathcal{B}_0, \mathcal{B}_1)$ is defined as $\operatorname{Ext}_{\mathcal{O}_X}^*(\iota_{0*}\mathcal{L}_0, \iota_{1*}\mathcal{L}_1)$ \footnote{More precisely, according to \cite{KatSha2002}, a B-brane $\mathcal{B}_i = (S_i, L_i)$ in our sense should correspond to the sheaf $\iota_{i*}\widetilde{\mathcal{L}}_i$ instead of $\iota_{i*}\mathcal{L}_i$, where $\widetilde{\mathcal{L}}_i$ is the sheaf of holomorphic sections of $L_i \otimes \sqrt{K_{S_i}}$, assuming a square root $\sqrt{K_{S_i}}$ of the canonical bundle of $S_i$ exists. When $S_i$ is a hyperK\"ahler manifold, $\sqrt{K_{S_i}}$ can be taken to be trivial.}, where $\iota_i: S_i \hookrightarrow X$ is the inclusion map and $\mathcal{L}_i$ is the sheaf of holomorphic sections of $L_i$ for $i \in \{0, 1\}$. By Theorem A.1 in \cite{CalKatSha2003}, we obtain the following local-to-global $\operatorname{Ext}$ spectral sequence
\begin{equation}
\label{Equation 4.1}
E_2 = H_{\overline{\partial}}^{0, p}\left(S, (\widetilde{N}^{1, 0})^{\wedge (q-m)} \otimes \det N^{1, 0}(S/S_1) \otimes \widetilde{L} \right) \Rightarrow \operatorname{Ext}_{\mathcal{O}_X}^{p+q}(\iota_{0*}\mathcal{L}_0, \iota_{1*}\mathcal{L}_1),
\end{equation}
where $\widetilde{N}$ is the \emph{excess bundle} for the intersection of $S_0$ and $S_1$, i.e.
\begin{equation*}
	\widetilde{N} = \frac{TX \vert_S}{TS_0 \vert_S + TS_1 \vert_S},
\end{equation*}
and $m$ is the rank of the holomorphic normal bundle $N^{1, 0}(S/S_1)$ of $S$ in $S_1$. Note that $S$, $S_0$ and $S_1$ are all hyperK\"ahler submanifolds of $X$. Similar to the proof of Proposition 5.1 in \cite{Ver1998}, we can show that $TX \vert_S$, $TS_0 \vert_S + TS_1 \vert_S \cong (TS_0 \vert_S \oplus TS_1 \vert_S)/TS$ and $\widetilde{N}$ are all hyperholomorphic bundles. Then the following short exact sequence of holomorphic vector bundles over $S$ splits:
\begin{center}
	\begin{tikzcd}
		0 \ar[r] & T^{1, 0}S_0 \vert_S + T^{1, 0}S_1 \vert_S \ar[r] & T^{1, 0}X \vert_S \ar[r] & \widetilde{N}^{1, 0} \ar[r] & 0
	\end{tikzcd}
\end{center}
by Proposition 5.2 in \cite{Ver1998}. Theorem 1.8 in \cite{AriCuaHab2019} together with arguments in Section 1.4 in the same reference reveals that the local-to-global $\operatorname{Ext}$ spectral sequence in (\ref{Equation 4.1}) degenerates at the $E_2$-page. Therefore, we can regard $\operatorname{Hom}(\mathcal{B}_0, \mathcal{B}_1)$ as the Dolbeault cohomology
\begin{equation*}
	H_{\overline{\partial}}^{0, *}\left(S, \left( \bigwedge \widetilde{N}^{1, 0} \right) \otimes \det N^{1, 0}(S/S_1) \otimes \widetilde{L} \right).
\end{equation*}
Then we outline how a desirable $\operatorname{Sp}(1)$-symmetry appears in this case. We can easily see that the hypercomplex structures on $X$ intertwines the Hermitian vector bundles
\begin{equation*}
E_\zeta := \left( \bigwedge T_{J_\zeta}^{*(0, 1)}S \right) \otimes \left( \bigwedge \widetilde{N}_{J_\zeta}^{1, 0} \right) \otimes \det N_{J_\zeta}^{1, 0}(S/S_1) \otimes \widetilde{L}.
\end{equation*}
Following the approach stated in Remark \ref{Remark 3.8}, we can prove that it also intertwines the Dolbeault Laplacians $\Delta_{\overline{\partial}_{J_\zeta}}$ on $\Gamma(S, E_\zeta)$ because all the bundles involved are hyperholomorphic. Therefore, there is an $\operatorname{Sp}(1)$-symmetry intertwining $\operatorname{Hom}_{J_\zeta}(\mathcal{B}_0, \mathcal{B}_1)$.\par
As a final remark, Verbitsky \cite{Ver1999} defined the tensor category of \emph{polystable hyperholomorphic reflexive sheaves} for a compact hyperK\"ahler manifold with respect to each complex structure in the twistor family, which should be interpreted as the category of (B, B, B)-branes. He showed that there exists a natural equivalence of any two of these tensor categories. This reveals that an $\operatorname{Sp}(1)$-symmetry for hyperbranes might appear even on the categorical level.

\subsection{$\operatorname{Sp}(1)$-symmetries for morphism spaces between (A, B, A)-branes}
\quad\par
Suppose $\mathcal{B}_0, \mathcal{B}_1$ are (A, B, A)-branes on $X$, i.e. hyperbranes on $X$ with respect to the twistor family $\{\mathcal{J}_\zeta\}_{\zeta \in \mathbb{S}^2}$ of generalized complex structures on $X$ determined by the $J$-holomorphic symplectic form $\Omega_J = \omega_K + \sqrt{-1}\omega_I$ as in Example \ref{Example A.8}. While it is unknown that our expectation is true in general, we place our observation on $\operatorname{Sp}(1)$-symmetries for morphism spaces $\operatorname{Hom}_{\mathcal{J}_\zeta}(\mathcal{B}_0, \mathcal{B}_1)$ in certain cases in the following table.

\begin{center}
	\def\arraystretch{1.5}
	\begin{tabular}{|c|c|c|c|c|}
		\hline
		$\mathcal{B}_0$ & $\mathcal{B}_1$ & $\operatorname{Hom}_{\mathcal{J}_\zeta}(\mathcal{B}_0, \mathcal{B}_1)$ & $\operatorname{Sp}(1)$-symmetry & Idea of proof\\
		\hline
		$\mathcal{B}_{\operatorname{cc}}$ & $\mathcal{B}_{\operatorname{cc}}$ & $H_{\overline{\partial}_{J_{-\kappa \cdot \zeta}}}^{0, *}(X)$ & exists & Theorem \ref{Theorem 3.1}\\
		\hline
		$\overline{\mathcal{B}}_{\operatorname{cc}}$ & $\mathcal{B}_{\operatorname{cc}}$ & $\operatorname{ind} \slashed{D}_{J_{-\kappa \cdot \zeta}}$ & exists & Theorem \ref{Theorem 1.1}\\
		\hline
		$\Omega_J$-Lagr. & $\Omega_J$-Lagr. & $H_d^*(S, \widetilde{L})$ if $\zeta \neq \mp \mathbf{j}$; & exists & Hodge-de Rham\\
		& & $H_{\overline{\partial}_{\pm J}}^{0, *}(S, \textstyle \bigwedge \widetilde{N}_{\pm J}^{1, 0} \otimes \widetilde{L})$ if $\zeta  = \mp \mathbf{j}$ & & Theorem\\
		\hline
		$\Omega_J$-Lagr. & $\mathcal{B}_{\operatorname{cc}}$ & undefined if $\zeta \neq \mp \mathbf{j}$ & unknown & unknown\\
		& & $H_{\overline{\partial}_{\pm J}}^{0, *}(S, \widetilde{L} \otimes \sqrt{K_{\pm J}})$ if $\zeta = \mp \mathbf{j}$ & & \\
		\hline
	\end{tabular}
\end{center}
Note that we shall make our observation in the above cases ignoring instanton corrections due to the following two facts:
\begin{enumerate}
	\item when either one of $\mathcal{B}_0$ and $\mathcal{B}_1$ is space-filling, there are no instanton corrections \cite{Her2012}; and
	\item when the hyperK\"ahler metric on $X$ is complete, there are no instanton corrections for a pair of $\Omega_J$-Lagrangian submanifolds as A-branes on $(X, \omega_{\cos \theta I + \sin \theta K})$ for a generic choice of $\theta \in [0, 2\pi)$ \cite{SolVer2019}. 
\end{enumerate}

Now, we give more detailed comments on the morphism spaces appeared in the above table.

\subsubsection{The $(\mathcal{B}_{\operatorname{cc}}, \mathcal{B}_{\operatorname{cc}})$-system}
\quad\par
Recall that the curvature $F_L$ of the Chan-Paton bundle $L$ of the space-filling A-brane $\mathcal{B}_{\operatorname{cc}}$ on $X$ satisfies $\tfrac{\sqrt{-1}}{2\pi} F_L = \omega_J$. We adopt the approach of Kapustin-Li \cite{KapLi2005} that the $\mathbb{S}^2$-indexed family of morphism spaces for this system is computed by Lie algebroid cohomology 
\footnote{In our convention, we take the Lie algebroid as the $\sqrt{-1}$-eignebundle of the generalized complex structure in the generalized tangent subbundle. Due to this convention, the generalized complex structures $\mathcal{J}_\mathbf{i}$, $\mathcal{J}_\mathbf{k}$ correspond to the symplectic forms $\omega_I, \omega_K$ respectively while $\mathcal{J}_{\mp \mathbf{j}}$ correspond to complex structures $\pm J$.}
associated to the underlying generalized complex submanifold. A careful calculation shows that for $\zeta \in \mathbb{S}^2$,
\begin{align*}
\operatorname{Hom}_{\mathcal{J}_\zeta}(\mathcal{B}_{\operatorname{cc}}, \mathcal{B}_{\operatorname{cc}}) = H_{\overline{\partial}_{J_{-\kappa \cdot \zeta}}}^{0, *}(X),
\end{align*}
where $\kappa = e^{-\frac{\pi}{4} \mathbf{j}}$. For example, $\operatorname{Hom}_J(\mathcal{B}_{\operatorname{cc}}, \mathcal{B}_{\operatorname{cc}}) = H_{\overline{\partial}_J}^{0, *}(X)$ and $\operatorname{Hom}_{\omega_K}(\mathcal{B}_{\operatorname{cc}}, \mathcal{B}_{\operatorname{cc}}) = H_{\overline{\partial}_I}^{0, *}(X)$.

\subsubsection{The $(\overline{\mathcal{B}}_{\operatorname{cc}}, \mathcal{B}_{\operatorname{cc}})$-system}
\quad\par
Recall that the Chan-Paton bundle of $\overline{\mathcal{B}}_{\operatorname{cc}}$ is the dual $L^\vee$ of $L$. For $\zeta = \mp \mathbf{j}$, the morphism space is given by $\operatorname{Hom}_{\pm J}(\overline{\mathcal{B}}_{\operatorname{cc}}, \mathcal{B}_{\operatorname{cc}}) = H_{\overline{\partial}_{\pm J}}^{0, *}(X, \operatorname{Hom}(L^\vee, L)) = H_{\overline{\partial}_{\pm J}}^{0, *}(X, L^{\otimes 2})$; for other $\zeta$, while we have not found a cohomological description of the morphism space, we propose the following definition via a choice of Hermitian metric.

\begin{definition}
	\label{Definition 4.2}
	Let $(M, \omega)$ be a compact symplectic manifold and $B$ be a closed real $2$-form on $M$. Let $\mathcal{B}_0, \mathcal{B}_1$ be space-filling A-branes on $(M, \omega, B)$ with Chan-Paton bundles $L_0, L_1$ respectively. Let $(I, \Omega)$ be the holomorphic symplectic structure on $M$ induced by $\mathcal{B}_1$. Suppose $g$ is a Hermitian metric on $(M, I)$. Define
	\begin{align*}
	\operatorname{Hom}_\mathcal{J}(\mathcal{B}_0, \mathcal{B}_1) = \operatorname{ind} \slashed{D},
	\end{align*}
	where $\mathcal{J}$ is the generalized complex structure on $M$ induced by $\omega$ and $B$ as in Example \ref{Example 4.2} and $\slashed{D}$ is the $\operatorname{Spin}^{\operatorname{c}}$-Dirac operator on $\Omega_I^{0, *}(M, \operatorname{Hom}(L_0, L_1))$.
\end{definition}

Adopting Definition \ref{Definition 4.2}, by Hodge theory we have $\operatorname{Hom}_{\mathcal{J}_\zeta}(\mathcal{B}_0, \mathcal{B}_1) = \operatorname{ind} \slashed{D}_{J_{-\kappa \cdot \zeta}}$ for all $\zeta \in \mathbb{S}^2$, where $\slashed{D}_{J_{-\kappa \cdot \zeta}}$ is the $\operatorname{Spin}^{\operatorname{c}}$-Dirac operator on $\Omega_{J_{-\kappa \cdot \zeta}}^{0, *}(X, L^{\otimes 2})$.

\subsubsection{The $(\mathcal{B}_0, \mathcal{B}_1)$-system for that $S_0, S_1$ are $\Omega_J$-Lagrangian submanifolds}
\quad\par
Suppose $S_0, S_1$ are $\Omega_J$-Lagrangian submanifolds of $X$. We first consider the B-model on $(X, J)$ (the case for $(X, -J)$ is similar). For each $i \in \{0, 1\}$, assume $S_i$ is spin, so that it admits a square root $\sqrt{K_i}$ of the canonical bundle $K_i$ of $S$. Then the B-brane $\mathcal{B}_i$ on $X$ corresponds to the derived pushforward of the sheaf of holomorphic sections of $L_i \otimes \sqrt{K_i}$ onto $X$. Let $N(S/S_i)$ be the normal bundle of $S$ in $S_i$. By Lemma 19 in \cite{Leu2002}, $\det N^{1, 0} (S/S_0) \otimes \det N^{1, 0} (S/S_0)$ is holomorphically trivial. Then the arguments in Section 5.3 in \cite{KatSha2002} show that
\begin{equation*}
	\operatorname{Hom}(K_0 \vert_S, K_1 \vert_S) \cong \det (N^{1, 0}(S/S_1)^\vee)^{\otimes 2}.
\end{equation*}
Hence, under a mild assumption that $\operatorname{Hom}(\sqrt{K_0} \vert_S, \sqrt{K_1} \vert_S) \cong \det (N^{1, 0}(S/S_1)^\vee)$, we obtain a local-to-global $\operatorname{Ext}$ spectral sequence similar to (\ref{Equation 4.1}) whose $E_2$-page is $H_{\overline{\partial}}^{0, p}(S, (\widetilde{N}^{1, 0})^{\wedge (q-m)} \otimes \widetilde{L})$, where $\widetilde{N}$ is the excess bundle for the intersection of $S_0$ and $S_1$ and $m$ is the rank of $N^{1, 0}(S/S_1)$. If either $S_0 = S_1$ or $S_0$ and $S_1$ intersect transversely, then this local-to-global $\operatorname{Ext}$ spectral sequence degenerates at the $E_2$ page. For our purpose, we define $\operatorname{Hom}_J(\mathcal{B}_0, \mathcal{B}_1)$ as $H_{\overline{\partial}}^{0, *}( S, \textstyle \bigwedge \widetilde{N}^{1, 0} \otimes \widetilde{L} )$.\par
On the other hand, consider the A-model on $(X, \mathcal{J}_\zeta)$, where $\zeta \in \mathbb{S}^2$ and $\zeta \neq \pm \mathbf{j}$. There exists a spectral sequence converging to the Floer homology for the pair $(S_0, S_1)$ of Lagrangian submanifolds whose $E_2$ page is the singular cohomology of $S$ with a certain coefficient in the Novikov field \cite{FukOhOhtOno2009}. For our purpose again, we define $\operatorname{Hom}_{\mathcal{J}_\zeta}(\mathcal{B}_0, \mathcal{B}_1)$ to be $H_d^*(S, \widetilde{L})$, where $d$ is the exterior covariant derivative on $\widetilde{L}$. In the sense of the above definitions, we can write down an $\operatorname{Sp}(1)$-symmetry intertwining the $\mathbb{S}^2$-indexed family of morphism spaces $\operatorname{Hom}_{\mathcal{J}_\zeta}(\mathcal{B}_0, \mathcal{B}_1)$, as we have the following isomorphisms
\begin{align*}
H_{\overline{\partial}_{\pm J}}^{0, *} ( S, \textstyle \bigwedge \widetilde{N}_{\pm J}^{1, 0} \otimes \widetilde{L} ) \cong H_{\overline{\partial}_{\pm J}}^{*, *} ( S, \widetilde{L} ) \cong H_d^*(S, \widetilde{L}),
\end{align*}
the first of which is induced by the isomorphism $\widetilde{N}_{\pm J}^{1, 0} \cong T_{\pm J}^{*(1, 0)}S$ via the $\pm J$-holomorphic symplectic form $\tfrac{1}{2} (\omega_K \pm \sqrt{-1}\omega_I)$ and the second of which is by Hodge-de Rham Theorem for unitary local systems (see for instance 
\cite{CheYan2019}).

\subsubsection{The $(\mathcal{B}_0, \mathcal{B}_{\operatorname{cc}})$-system for that $S_0$ is a $\Omega_J$-Lagrangian submanifold}
\quad\par
Let $S_0$ be a $\Omega_J$-Lagrangian submanifold of $X$ and $\mathcal{B}_1 = \mathcal{B}_{\operatorname{cc}}$. Assume $S = S_0$ admits spin structures. Again, we first consider the B-model on $(X, J)$ and choose a square root $\sqrt{K_J}$ of the canonical bundle of $S$. The morphism space $\operatorname{Hom}_J(\mathcal{B}_0, \mathcal{B}_{\operatorname{cc}})$ is defined as $\operatorname{Ext}_{\mathcal{O}_X}(\iota_*\widetilde{\mathcal{L}}_0, \mathcal{L}_1)$, where $\iota: S \hookrightarrow X$ is the inclusion map, $\widetilde{\mathcal{L}}_0$ and $\mathcal{L}_1$ are the sheaves of holomorphic sections of $L_0 \otimes \sqrt{K_J}$ and $L_1$ respectively. As $S$ is a holomorphic Lagrangian submanifold of $X$, its holomorphic normal bundle in $X$ is isomorphic to $K_J^\vee$. Thus, $\operatorname{Hom}_J(\mathcal{B}_0, \mathcal{B}_{\operatorname{cc}}) \cong H_{\overline{\partial}_J}^{0, *}(S, \widetilde{L} \otimes \sqrt{K_J})$. Similarly, $\operatorname{Hom}_{-J}(\mathcal{B}_0, \mathcal{B}_{\operatorname{cc}}) \cong H_{\overline{\partial}_{-J}}^{0, *}(S, \widetilde{L} \otimes \sqrt{K_{-J}})$, where $\sqrt{K_{-J}}$ is chosen to be the dual of $\sqrt{K_J}$. These two B-model morphism spaces are isomorphic via the Hodge star operator on $S$.\par
For $\zeta \in \mathbb{S}^2$ other than $\pm \mathbf{j}$, there is an attempt in \cite{BisGua2022} to mathematically define $\operatorname{Hom}_{\omega_K}(\mathcal{B}_0, \mathcal{B}_{\operatorname{cc}})$ in the presence of an $I$-holomorphic Lagrangian fibration of $X$. We believe that there is an appropriate mathematical definition of the A-model morphism spaces such that our expectation is true for the $(\mathcal{B}_0, \mathcal{B}_{\operatorname{cc}})$-system, at least under suitable conditions on the formal neighbourhood of $S$ in $X$.

\appendix

\section{Recollections in Generalized Geometry}
\label{Appendix A}
In this appendix, we shall review some relevant basic knowledges in generalized geometry, which is a subject proposed by Hitchin \cite{Hit2003} and whose foundation is first developed by Gualtieri \cite{Gua2011}.\par
Throughout this appendix, assume $M$ is a smooth manifold. Generalized geometry concerns with different kinds of geometric structures on the \emph{generalized tangent bundle} $\mathcal{T}M := TM \oplus T^*M$. This bundle $\mathcal{T}M$ admits canonically
\begin{enumerate}
	\item a pseudo-Riemannian metric $\langle \quad, \quad \rangle$ given by $\langle u + \alpha, v + \beta \rangle = \tfrac{1}{2} (\alpha(v) + \beta(u))$ for $x \in M$ and $u, v \in T_xM$ and $\alpha, \beta \in T_x^*M$; and
	\item an $\mathbb{R}$-bilinear map $[\quad, \quad]: \Gamma(M, \mathcal{T}M) \times \Gamma(M, \mathcal{T}M) \to \Gamma(M, \mathcal{T}M)$ given by
	\begin{align*}
		[u + \alpha, v + \beta] = [u, v] + \mathcal{L}_u\beta - \iota_vd\alpha,
	\end{align*}
	for $u, v \in \Gamma(M, TM)$ and $\alpha, \beta \in \Gamma(M, T^*M)$, where $[u, v]$ is the Lie bracket of vector fields.
\end{enumerate}
The data $\mathcal{T}M$, $\langle \quad, \quad \rangle$, $[\quad, \quad]$ together with the canonical projection $\mathcal{T}M \to TM$ forms an exact Courant algebroid over $M$ with vanishing \v{S}evera class. We denote by $\operatorname{O}(\mathcal{T}M)$ the fibre bundle whose fibre $\operatorname{O}(\mathcal{T}_xM)$ over $x \in M$ is the set of all $\phi \in \operatorname{End}(\mathcal{T}_xM)$ such that $\langle \phi(e), \phi(e') \rangle = \langle e, e'\rangle$ for all $e, e' \in \mathcal{T}_xM$. A subbundle $L$ of $\mathcal{T}M_\mathbb{C}$ is said to be \emph{involutive} if $[\Gamma(M, L), \Gamma(M, L)] \subset \Gamma(M, L)$.\par
Now we shall introduce generalized complex structures, generalized hypercomplex structures and generalized hyperK\"ahler structures in the following subsections.

\subsection{Generalized complex geometry}
\quad\par
A complex structure on $M$ is a smooth section $J$ of $\operatorname{End}(TM)$ such that $J^2 = -\operatorname{Id}_{TM}$ and the $\sqrt{-1}$-eignebundle of $J$ in $TM_\mathbb{C}$ is involutive. A generalized complex structure is defined similarly.

\begin{definition}
	A \emph{generalized complex structure} on $M$ is a smooth section $\mathcal{J}$ of $\operatorname{O}(\mathcal{T}M)$ such that $\mathcal{J}^2 = -\operatorname{Id}_{\mathcal{T}M}$ and the $\sqrt{-1}$-eigenbundle of $\mathcal{J}$ in $\mathcal{T}M_\mathbb{C}$ is involutive. A \emph{generalized complex manifold} is a smooth manifold equipped with a generalized complex structure.
\end{definition}

Then we introduce a generalization of complex submanifolds.

\begin{definition}
	A \emph{generalized complex submanifold} \footnote{Here we adopt Gualtieri's definition \cite{Gua2011} of generalized complex submanifolds. However, in some literatures \cite{BenBoy2004, Cav2007}, generalized complex submanifolds refer to other types of submanifolds while generalized Lagrangian submanifolds refer to such submanifolds together with $2$-forms appeared in our definition.} of a generalized complex manifold $(M, \mathcal{J})$ is a pair $(S, F)$, where $S$ is a smooth submanifold of $M$ and $F \in \Omega^2(S)$ such that the \emph{generalized tangent subbundle} of $(S, F)$
	\begin{equation*}
		\mathcal{T}^F S := \{ u + \alpha \in T_xS \oplus T_x^*M: x \in M, \alpha \vert_{T_xS} = \iota_uF \}
	\end{equation*}
	is preserved by $\mathcal{J}$, i.e. $\mathcal{J}(\mathcal{T}^F S) = \mathcal{T}^F S$.
\end{definition}

Branes, serving as boundary conditions in a $2$d $\sigma$-model with target on a generalized complex manifold, are generalized complex submanifolds such that their underlying $2$-forms represent integral cohomology classes.

\begin{definition}
	\label{Definition A.3}
	A (\emph{rank}-$1$) \emph{brane} on a generalized complex manifold $(M, \mathcal{J})$ is a submanifold $S$ of $M$ together with a Hermitian line bundle $L$ over $S$ with a unitary connection $\nabla$, called its \emph{Chan-Paton bundle}, such that $(S, F)$ is a generalized complex submanifold of $(M, \mathcal{J})$, where $F = \tfrac{\sqrt{-1}}{2\pi} F_L$ and $F_L$ is the curvature of $\nabla$.
\end{definition}

The following are two basic examples of generalized complex structures.

\begin{example}
	\label{Example 4.2}
	Suppose $\omega$ is a symplectic form and $B$ is a closed real $2$-form on $M$. Then the following vector bundle endomorphism is a generalized complex structure on $M$:
	\begin{align*}
	\mathcal{J}_{\omega, B} := \begin{pmatrix}
	\operatorname{Id}_{TM} & 0\\
	B_\sharp & \operatorname{Id}_{T^*M}
	\end{pmatrix}
	\begin{pmatrix}
	0 & -(\omega^{-1})^\sharp\\
	\omega_\sharp & 0
	\end{pmatrix}
	\begin{pmatrix}
	\operatorname{Id}_{TM} & 0\\
	-B_\sharp & \operatorname{Id}_{T^*M}
	\end{pmatrix}
	: \mathcal{T}M \to \mathcal{T}M.
	\end{align*}
	Branes on $(M, \mathcal{J}_{\omega, B})$ are called \emph{A-branes} on $(M, \omega, B)$. Explicitly, an A-brane on $(M, \omega, B)$ is a coisotropic submanifold $S$ of $M$ equipped with a Hermitian line bundle $L$ over it with a unitary connection of curvature $F_L$ satisfying the conditions that, letting $F = \tfrac{\sqrt{-1}}{2\pi}F_L$, $\mathcal{L}S$ be the characteristic foliation of $S$ and $\mathcal{F}S = TS/\mathcal{L}S$,
	\begin{itemize}
		\item $F - B \vert_S \in \Omega^2(S)$ descends to a smooth section $\tilde{F} \in \Gamma(S, \bigwedge^2 \mathcal{F}^*S)$; and
		\item the composition $I := \tilde{\omega}^{-1} \circ \tilde{F}: \mathcal{F}S \to \mathcal{F}S$ defines a complex structure on the vector bundle $\mathcal{F}S$, where $\tilde{\omega} \in \Gamma(S, \bigwedge^2 \mathcal{F}^*S)$ is induced by $\omega \vert_S$.
	\end{itemize}
	In particular, when an A-brane on $M$ is \emph{space-filling}, i.e. its underlying submanifold is $M$, it induces a holomorphic symplectic structure $(I, \Omega)$ on $M$ with $\Omega := F - B + \sqrt{-1} \omega$.
\end{example}

\begin{example}
	Suppose $J$ is a complex structure on $M$ and $\sigma$ is a $\overline{\partial}$-closed $(0, 2)$-form on $M$. Then the following vector bundle endomorphism is a generalized complex structure on $M$:
	\begin{align*}
	\mathcal{J}_{J, \sigma} := \begin{pmatrix}
	-J & 0\\
	\omega_\sharp & J^{\operatorname{t}}
	\end{pmatrix}: \mathcal{T}M \to \mathcal{T}M, \quad \text{where } \omega = 2\sqrt{-1}(\sigma - \overline{\sigma}).
	\end{align*}
	Branes on $(M, \mathcal{J}_{J, \sigma})$ are called \emph{B-branes} on $(M, J, \sigma)$. Explicitly, a B-brane on $(M, J, \sigma)$ is a complex submanifold $S$ of $(M, J)$ equipped with a Hermitian line bundle $L$ over it with a unitary connection of curvature $F_L$ satisfying the condition that the $(0, 2)$-part of $F = \tfrac{\sqrt{-1}}{2\pi}F_L$ is $\sigma \vert_S$.
\end{example}

\subsection{Generalized hypercomplex geometry}
\label{Subsection A.2}
\quad\par
A hypercomplex structure on $M$ is equivalent to an algebra homomorphism
\begin{equation*}
	\mathfrak{J}: \mathbb{H} \to \Gamma(M, \operatorname{End}(TM))
\end{equation*}
such that for all $\zeta \in \mathbb{S}^2$, $J_\zeta := \mathfrak{J}(\zeta)$ is a complex structure on $M$, where $\mathbb{S}^2$ denotes the unit sphere of the imaginary part of the algebra of quaternions $\mathbb{H}$. We have an analogous definition for generalized hypercomplex structures.

\begin{definition}
	A \emph{generalized hypercomplex structure} on $M$ is an algebra homomorphism
	\begin{align*}
		\mathcal{J}: \mathbb{H} \to \Gamma(M, \operatorname{End}(\mathcal{T}M)) \quad \zeta \to \mathcal{J}_\zeta
	\end{align*}
	such that for all $\zeta \in \mathbb{S}^2$, $\mathcal{J}_\zeta$ is a generalized complex structure of $M$. A \emph{generalized hypercomplex manifold} is a smooth manifold equipped with a generalized hypercomplex structure.
\end{definition}

For a generalized hypercomplex manifold $(M, \mathcal{J})$, we call the $\mathbb{S}^2$-indexed family $\{ \mathcal{J}_\zeta \}_{\zeta \in \mathbb{S}^2}$ the \emph{twistor family of generalized complex structures} of $(M, \mathcal{J})$.
Analogous to generalized complex geometry, the following are two basic examples of generalized hypercomplex structures \cite{HonSti2015}. Let $(1, \mathbf{i}, \mathbf{j}, \mathbf{k})$ be the standard basis of $\mathbb{H}$.

\begin{example}
	\label{Example A.7}
	A hypercomplex structure $(I, J, K)$ on $M$ induces a generalized hypercomplex structure $\mathcal{J}$ determined by
	\begin{equation*}
		\mathcal{J}_\mathbf{i} = \begin{pmatrix}
			I & 0\\
			0 & -I^{\operatorname{t}}
		\end{pmatrix}, \quad
		\mathcal{J}_\mathbf{j} = \begin{pmatrix}
		J & 0\\
		0 & -J^{\operatorname{t}}
		\end{pmatrix}, \quad
		\mathcal{J}_\mathbf{k} = \begin{pmatrix}
		K & 0\\
		0 & -K^{\operatorname{t}}
		\end{pmatrix}.
	\end{equation*}
\end{example}

\begin{example}
	\label{Example A.8}
	Let $J$ be a complex structure on $M$ and $\Omega$ be a holomorphic symplectic form on $(M, J)$. Then these structures induce a generalized hypercomplex structure $\mathcal{J}$ determined by
	\begin{equation*}
	\mathcal{J}_\mathbf{i} = \begin{pmatrix}
	0 & -(\omega_I^{-1})^\sharp\\
	(\omega_I)_\sharp & 0
	\end{pmatrix}, \quad
	\mathcal{J}_\mathbf{j} = \begin{pmatrix}
	J & 0\\
	0 & -J^{\operatorname{t}}
	\end{pmatrix}, \quad
	\mathcal{J}_\mathbf{k} = \begin{pmatrix}
	0 & -(\omega_K^{-1})^\sharp\\
	(\omega_K)_\sharp & 0
	\end{pmatrix},
	\end{equation*}
	where $\omega_K = \operatorname{Re} \Omega$ and $\omega_I = \operatorname{Im} \Omega$.
\end{example}

\subsection{Generalized hyperK\"ahler geometry}
\quad\par
We first introduce a generalization of Riemannian metrics.

\begin{definition}
	A \emph{generalized Riemannian metric} on $M$ is a smooth section $\mathcal{G}$ of $\operatorname{O}(\mathcal{T}M)$ such that $\mathcal{G}^2 = \operatorname{Id}_{\mathcal{T}M}$ and $\langle \mathcal{G}(\quad), \quad \rangle$ is a positive definite metric on $\mathcal{T}M$.
\end{definition}

There is a one-to-one correspondence between generalized Riemannian metrics $\mathcal{G}$ on $M$ and pairs $(g, B)$, where $g$ is a Riemannian metric on $M$ and $B$ is a real closed $2$-form on $M$, and the correspondence is given by
\begin{equation*}
	\mathcal{G} = \begin{pmatrix}
	\operatorname{Id}_{TM} & 0\\
	B_\sharp & \operatorname{Id}_{T^*M}
	\end{pmatrix}
	\begin{pmatrix}
	0 & (g^{-1})^\sharp\\
	g_\sharp & 0
	\end{pmatrix}
	\begin{pmatrix}
	\operatorname{Id}_{TM} & 0\\
	-B_\sharp & \operatorname{Id}_{T^*M}
	\end{pmatrix}.
\end{equation*}
Moreover, $\mathcal{G}$ determines an algebra homomorphism
\begin{equation*}
	\mathbb{D} \to \Gamma(M, \operatorname{End}(\mathcal{T}M)), \quad a + b \mathbf{r} \mapsto a \operatorname{Id}_{\mathcal{T}M} + b\mathcal{G},
\end{equation*}
where $\mathbb{D}$ is the algebra of \emph{split complex numbers}, i.e. the $2$-dimensional algebra over $\mathbb{R}$ with identity $1$ and an element $\mathbf{r}$ such that $\mathbf{r}^2 = 1$ and $(1, \mathbf{r})$ forms a basis. This motivates our definition of a generalized hyperK\"ahler structure on $M$ as a certain algebra homomorphism from the algebra of \emph{split-biquaternions} $\mathbb{H} \otimes_\mathbb{R} \mathbb{D}$ to $\Gamma(M, \operatorname{End}(\mathcal{T}M))$.

\begin{definition}
	A \emph{generalized hyperK\"ahler structure} on $M$ is an algebra homomorphism
	\begin{align*}
	\mathcal{K}: \mathbb{H} \otimes_\mathbb{R} \mathbb{D} \to \Gamma(M, \operatorname{End}(\mathcal{T}M)), \quad \mu \mapsto \mathcal{K}_\mu,
	\end{align*}
	such that $\mathcal{G} := \mathcal{K}_{1 \otimes \mathbf{r}}$ is a generalized Riemannian metric on $M$ and for all $\zeta \in \mathbb{S}^2$, $\mathcal{J}_\zeta := \mathcal{K}_{\zeta \otimes 1}$ and $\tilde{\mathcal{J}}_\zeta := \mathcal{K}_{\zeta \otimes \mathbf{r}} = \mathcal{J}_\zeta\mathcal{G} = \mathcal{G}\mathcal{J}_\zeta$ are generalized complex structures on $M$. A \emph{generalized hyperK\"ahler manifold} is a smooth manifold equipped with a generalized hyperK\"ahler structure.
\end{definition}

We call $\mathcal{K} \circ \iota: \mathbb{H} \to \Gamma(M, \operatorname{End}(\mathcal{T}M))$ the underlying generalized hypercomplex structure of $\mathcal{K}$, where $\iota: \mathbb{H} \hookrightarrow \mathbb{H} \otimes_\mathbb{R} \mathbb{D}$ is given by $\zeta \mapsto \zeta \otimes 1$.\par
To see how hyperK\"ahler structures give rise to generalized hyperK\"ahler structures, we quote the following proposition.

\begin{proposition}
	\label{Proposition A.11}
	(Proposition 8 in \cite{Des2021})
	There is a one-to-one correspondence between generalized hyperK\"ahler structures $\mathcal{K}$ on $M$ and sets of data
	\begin{equation*}
		(g, B, \mathfrak{J}^+, \mathfrak{J}^-),
	\end{equation*}
	where $B$ is a real closed $2$-form on $M$, $\mathfrak{J}^+, \mathfrak{J}^-$ are hypercomplex structures on $M$ and $g$ is a hyperK\"ahler metric on $M$ with respect to both $\mathfrak{J}^+$ and $\mathfrak{J}^-$.
\end{proposition}

We know that the automorphism group of the $\mathbb{R}$-algebra $\mathbb{H}$, which is isomorphic to $\operatorname{SO}(3)$, naturally acts on the space of hyperK\"ahler structures on $M$ by rotating the twistor families of complex structures. Indeed, there is an $\mathbb{R}$-algebra isomorphism
\begin{equation*}
\mathbb{H} \oplus \mathbb{H} \to \mathbb{H} \otimes_\mathbb{R} \mathbb{D}, \quad (\zeta_+, \zeta_-) \mapsto \zeta_+ \otimes \mathbf{r}_+ + \zeta_- \otimes \mathbf{r}_-,
\end{equation*}
where $\mathbf{r}_\pm = \tfrac{1}{2} (1 \pm \mathbf{r})$, and Proposition \ref{Proposition A.11} implies that the $\operatorname{SO}(3) \times \operatorname{SO}(3)$-action on $\mathbb{H} \oplus \mathbb{H} \cong \mathbb{H} \otimes_\mathbb{R} \mathbb{D}$ induces a natural $\operatorname{SO}(3) \times \operatorname{SO}(3)$-action on the space of generalized hyperK\"ahler structures on $M$: if $(\phi_+, \phi_-) \in \operatorname{SO}(3) \times \operatorname{SO}(3)$ and $\mathcal{K}$ corresponds to $(g, B, \mathfrak{J}^+, \mathfrak{J}^-)$, then $(\phi_+, \phi_-) \cdot \mathcal{K}$ corresponds to $(g, B, \phi_+ \cdot \mathfrak{J}^+, \phi_- \cdot \mathfrak{J}^-)$. As a result, for a generalized hyperK\"ahler structure $\mathcal{K}$ on $M$, the restriction of $\mathcal{K}$ on $\mathbb{S}^2 \times \mathbb{S}^2 \subset \mathbb{H} \oplus \mathbb{H} \cong \mathbb{H} \otimes_\mathbb{R} \mathbb{D}$ is an $\mathbb{S}^2 \times \mathbb{S}^2$-indexed family of generalized complex structures.

\begin{remark}
	\label{Remark A.12}
	A single hyperK\"ahler structure $(g, \mathfrak{J})$ on $M$ induces a generalized hyperK\"ahler structure $\mathcal{K}$ associated to $(g, 0, \mathfrak{J}, \mathfrak{J})$. In this case, the aforementioned $\mathbb{S}^2 \times \mathbb{S}^2$-indexed family is called the \emph{generalized twistor family} of the hyperK\"ahler manifold $(M, g, \mathfrak{J})$ in \cite{GloSaw2015}, and the $\operatorname{SO}(3) \times \operatorname{SO}(3)$-action on $\mathcal{K}$ gives rise to a family of generalized hyperK\"ahler structures on $M$ whose underlying generalized hypercomplex structures are induced by either a hypercomplex structure or a holomorphic symplectic form on $M$, up to B-field transformation and rescaling of the metric.
\end{remark}

\section{Irreducible Representations of $\operatorname{SL}(2, \mathbb{C})$}
\label{Appendix B}
In this appendix, we shall recall the description of each finite dimensional irreducible complex representation of $\operatorname{SL}(2, \mathbb{C})$ in terms of polynomials of two variables. Consider any $m \in \mathbb{N}$. Let $U_m$ be the vector space of homogeneous polynomials of degree $m$ in two variables $x, y$ over $\mathbb{C}$. This vector space $U_m$ has a basis $(x^m, x^{m-1}y, ..., xy^{m-1}, y^m)$. Define three linear operators on $U_m$:
\begin{equation}
\label{Equation B.1}
L_m := x \tfrac{\partial}{\partial y}, \quad \Lambda_m := y \tfrac{\partial}{\partial x}, \quad \text{and} \quad H_m := x \tfrac{\partial}{\partial x} - y \tfrac{\partial}{\partial y}.
\end{equation}
Then $(L_m, \Lambda_m, H_m)$ forms a Lefschetz triple and gives an irreducible complex representation
\begin{align*}
\rho_m^{\mathfrak{sl}(2, \mathbb{C})}: \mathfrak{sl}(2, \mathbb{C}) \to \mathfrak{gl}(U_m)
\end{align*}
of $\mathfrak{sl}(2, \mathbb{C})$ which integrates to an irreducible complex representation of $\operatorname{SL}(2, \mathbb{C})$:
\begin{equation*}
\rho_m^{\operatorname{SL}(2, \mathbb{C})}: \operatorname{SL}(2, \mathbb{C}) \to \operatorname{GL}(U_m).
\end{equation*}
Identify $\mathfrak{sp}(1)$ as the imaginary part of the algebra of quaternions, which is spanned by $\mathbf{i}, \mathbf{j}, \mathbf{k}$. There is a complex representation of $\mathfrak{sp}(1)$:
\begin{equation*}
\rho_m^{\mathfrak{sp}(1)}: \mathfrak{sp}(1) \to \mathfrak{gl}(U_m)
\end{equation*}
defined by $\rho_m^{\mathfrak{sp}(1)}(\mathbf{j}) = \sqrt{-1}H_m$, $\rho_m^{\mathfrak{sp}(1)}(\mathbf{k}) = L_m - \Lambda_m$ and $\rho_m^{\mathfrak{sp}(1)}(\mathbf{i}) = \sqrt{-1}(L_m + \Lambda_m)$. It integrates to a complex representation of $\operatorname{Sp}(1)$:
\begin{equation}
\label{Equation B.2}
\rho_m^{\operatorname{Sp}(1)}: \operatorname{Sp}(1) \to \operatorname{GL}(U_m)
\end{equation}
and for all $i \in \{0, ..., m\}$,
\begin{align}
\rho_m^{\operatorname{Sp}(1)}(\mathbf{j})(x^iy^{m-i}) = & \sqrt{-1}^{2i-m} x^iy^{m-i},\\
\label{Equation B.4}
\rho_m^{\operatorname{Sp}(1)}(\mathbf{k})(x^iy^{m-i}) = & (-1)^i x^{m-i}y^i,\\
\rho_m^{\operatorname{Sp}(1)}(\mathbf{i})(x^iy^{m-i}) = & \sqrt{-1}^m x^{m-i}y^i.
\end{align}

\bibliographystyle{amsplain}
\bibliography{References}

\end{document}